\documentclass[a4paper]{article}

\usepackage[english]{babel}

\usepackage{cite}

\usepackage[utf8]{inputenc}
\usepackage[T1]{fontenc}
\usepackage{lmodern}

\usepackage{hyperref}
\usepackage{amssymb,amsmath,amsthm}

\usepackage{todonotes}
\usepackage{enumerate}
\usepackage{mathtools,bbm}
\usepackage[normalem]{ulem}
\usepackage{cancel}
\usepackage{verbatim}
\usepackage{tikz}
\usetikzlibrary{matrix}

\usepackage{a4}
\usepackage{multirow}
\usepackage[footnotesize,bf,centerlast]{caption}
\usepackage[small,euler-digits,icomma,OT1,T1]{eulervm}

\usepackage{xcolor,colortbl}
\definecolor{Gray}{gray}{0.80}
\definecolor{LightGray}{gray}{0.90}

\newcommand{\cD}{\mathcal{D}}

\newcommand{\cH}{\mathcal{H}}

\newcommand{\cK}{\mathcal{K}}

\newcommand{\cM}{\mathcal{M}}

\newcommand{\cQ}{\mathcal{Q}}
\newcommand{\cR}{\mathcal{R}}

\newcommand{\cU}{\mathcal{U}}
\newcommand{\cV}{\mathcal{V}}

\newcommand{\cX}{\mathcal{X}}
\newcommand{\cY}{\mathcal{Y}}

\newcommand{\fL}{\mathfrak{L}}

\newcommand{\fP}{\mathfrak{P}}

\newcommand{\bC}{\mathbb{C}}

\newcommand{\bQ}{\mathbb{Q}}
\newcommand{\bR}{\mathbb{R}}

\newcommand{\bZ}{\mathbb{Z}}

\newcommand{\bONE}{\mathbbm{1}}
\newcommand{\Ba}{Ba}

\newcommand{\dd}{ \mathrm{d}}

\DeclareMathOperator*{\LIM}{LIM} 
\DeclareMathOperator*{\subLIM}{subLIM}
\DeclareMathOperator*{\superLIM}{superLIM}
\DeclareMathOperator*{\LIMSUP}{LIM \, SUP}
\DeclareMathOperator*{\LIMINF}{LIM \, INF}

\renewcommand{\epsilon}{\varepsilon}

\newcommand{\vn}[1]{\left| \! \left| #1\right| \! \right|}

\newcommand{\ip}[2]{\langle #1,#2\rangle}

\numberwithin{equation}{section}

\newtheorem{theorem}{Theorem}[section]
\newtheorem{lemma}[theorem]{Lemma}
\newtheorem{proposition}[theorem]{Proposition}

\theoremstyle{definition}
\newtheorem{definition}[theorem]{Definition}

\newtheorem{remark}[theorem]{Remark}

\newtheorem{assumption}[theorem]{Assumption}
\newtheorem{example}[theorem]{Example}
\newtheorem{condition}[theorem]{Condition}

\title{A general convergence result for viscosity solutions of Hamilton-Jacobi equations and non-linear semigroups}

\author{Richard C. Kraaij\thanks{Delft Institute of Applied Mathematics, Delft University of Technology, Van Mourik Broekmanweg 6, 2628 XE Delft, The Netherlands. \emph{E-mail address}: r.c.kraaij@tudelft.nl}}
\date{\today}

\begin{document}

\maketitle

\begin{abstract}
We extend the Barles-Perthame procedure \cite{BaPe88,BaPe90} (see also \cite{FK06}) of semi-relaxed limits of viscosity solutions of Hamilton-Jacobi equations of the type $f - \lambda H f = h$. 

The convergence result allows for equations on a `converging sequence of spaces' as well as Hamilton-equations written in terms of two equations in terms of operators $H_\dagger$ and $H_\dagger$ that serve as natural upper and lower bounds for the `true'  operator $H$.

In the process, we establish a strong relation between non-linear pseudo-resolvents and viscosity solutions of Hamilton-Jacobi equations. As a consequence we derive a convergence result for non-linear semigroups.

\noindent \emph{Keywords: Hamilton-Jacobi equation; \and viscosity solutions; \and Barles-Perthame method; non-linear semigroups}

\noindent \emph{MSC2010 classification: 49J45, 49L25, 47H20} 
\end{abstract}


\section{Introduction}

	In this paper, we will study the relation between three of the major objects in the field of semigroup theory: the semigroup, the resolvent and the generator. 
	
	\smallskip

	Consider the following problem posed by \cite{Ca1821}. Find all maps $f : \bR^+ \rightarrow \bC$ satisfying
	\begin{equation*}
	\begin{cases}
	\phi(t+s) = \phi(t)\phi(s) & \text{for all } s, t \geq 0, \\
	\phi(0) = 1.
	\end{cases}
	\end{equation*}
	Assuming that $\phi$ is continuous (the conclusion holds under much weaker assumptions), it can be shown that all maps of this type are of the form $\phi_a(t) := e^{ta} = \lim_{k \rightarrow \infty} \left(1 - \tfrac{t}{k}a \right)^{-k}$ with $a \in \bC$. 
	
	\smallskip
	
	The factor $a$, which can be found by $a = \partial_t \phi_a(t) |_{t=0}$, captures all essential information of the semigroup $\phi_a$. In addition, the dependence of $\phi_a$ on $a$ is robust under convergence: for a sequence of $a_n \in \bC$ with $a_n \rightarrow a$, it holds that $\phi_{a_n} \rightarrow \phi_{a}$ uniformly on compacts. 
	
	\smallskip
	
	Semigroup theory generalizes these three concepts to the level of semigroups on Banach and locally convex spaces. We will focus here on non-linear semigroups on the space of bounded measurable functions $M_b(X)$ on some space $X$. The three objects of interest are
	\begin{enumerate}[S(a)]
		\item A generator $H \subseteq M_b(X) \times M_b(X)$;
		\item A resolvent $R(\lambda) = (\bONE-\lambda H)^{-1}$, $\lambda > 0$;
		\item A semigroup $V(t) = \lim_{k \rightarrow \infty} R\left(\tfrac{t}{k}\right)^{k}$, $t \geq 0$.
	\end{enumerate}
In addition, one wants to establish relations between $H_n \rightarrow H$, $R_n(\lambda) \rightarrow R(\lambda)$ and $V_n(t) \rightarrow V(t)$.

\smallskip

In the context of linear semigroups on some Banach space $Y$, these results are all well known, cf. \cite{EN00,Yo78}. Two main results in this context are the Hille-Yosida generation theorem relating S(a), S(b) and S(c), whereas the Trotter-Kato-Kurtz approximation theorem establishes various implications between convergence of these three objects. The non-linear context is more complicated, cf. \cite{Mi92}. An analog of the Hille-Yosida theorem was proven by \cite{CL71} and the result by \cite{Ku73} is sufficiently general to do approximation theory in this context.

These results, however, assume that the equation $f - \lambda H f = h$ can be solved in the classical sense. While in the linear context this is often possible, in the non-linear context this is troublesome. \cite{CL83} introduced viscosity solutions that can replace classical solutions to achieve the same goals when restricting to operators taking values in the space of continuous functions.
The first stability result was observed already in this first paper \cite{CL83}, but it was Barles and Perthame \cite{BaPe88,BaPe90} who realized that if $H_n$ are operators converging to $H$, then semi-relaxed limits of viscosity solutions to Hamilton-Jacobi equations for $H_n$ yield viscosity sub- and supersolutions for the Hamilton-Jacobi equation for $H$. This method has been subsequently used in various papers, see \cite{CIL92} for a short discussion on the initial papers on this topic and see \cite[Section II.6]{FlSo06} for a somewhat more recent account in the setting of controlled Markov processes.

Finally, we want to mention \cite[Chapters 6 and 7]{FK06}, in which stability of solutions to Hamilton-Jacobi equations are studied in the context of large deviations for Markov processes, in which three important generalizations have been carried out with respect to the classical Barles-Perthame procedure:
\begin{enumerate}[(1)]
	\item Instead of working on a single space $X$, a sequence of spaces $X_n$ that are mapped into $X$ are considered. Conditions are given that imply the convergence of viscosity solutions of $f - \lambda H_n f = h_n$ on a space $X_n$ to a viscosity solution of $f - \lambda H f = h$ on $X$. For this result, Feng and Kurtz work with a generalized notion of buc (bounded and uniform on compacts) convergence that applies to functions on different spaces.
	\item Instead of working with a limiting operator $H$, \cite{FK06,Fe06} follow initial papers for Hamilton-Jacobi equations on infinite dimensional spaces, see \cite{CrLi94,Ta92,Ta94}, allowing for the possibility for a relaxed upper bound $H_\dagger$ and a relaxed lower bound $H_\ddagger$. Thus, in the limit a subsolution to $f - \lambda H_\dagger f = h$ and a supersolution to $f - \lambda H_\ddagger f = h$ is obtained. 
	\item The operators $H_\dagger$ and $H_\ddagger$ can take their images in the space of measurable functions $M(Y)$ on a space $Y$ instead of $X$, where $Y$ is a space containing more information than $X$, allowing for the use of the approximation theory in the setting of e.g. homogenization and multi-scale systems.
\end{enumerate}

Various applications of these extended methods are given in Chapters 11, 12 and 13 of \cite{FK06} and \cite{Fe06} and have recently been applied in large deviation theory in various settings \cite{DFL11,Po18,FeFoKu12,CoKr18,KrMa18,KrReVe18}.

\smallskip

The methods of \cite{FK06}, however, have two major drawbacks and lack one desirable property.
\begin{itemize}
		\item The extension of \cite{FK06} in their Chapter 7, which includes the generalizations (1), (2) and (3), is based on the property that the Hamiltonians $H_n$ are given in terms of an exponential tilt of an operator $A_n$ which is the generator of a Markov process. \cite{FK06} then approximate $A_n$ by its bounded Yosida approximant $A_n^\varepsilon = A_n(\bONE - \varepsilon A_n)^{-1}$. This leads to a continuous operator $H_n^\varepsilon$ that is easier to treat.
		A replacement of $H_n$ by $H_n^\varepsilon$ is not possible if $H_n$ is not derived from a linear operator $A_n$, therefore making it impossible to widely use the stability result in a general setting, excluding e.g. an application in the context of Gamma convergence \cite{DM93,Br02}.
		\item A second major drawback arises from the realization that in general, and in particular in infinite dimensions cf. \cite{Ta92,Ta94,CrLi94,Fe06,FeKa09,AmFe14,FeMiZi18}, it is advantageous to work with an upper and lower bound $H_\dagger$ and $H_\ddagger$ instead of a single Hamiltonian $H$. Therefore, instead of working with operators $H_n$ to obtain a limiting upper and lower bounds $H_\dagger, H_\ddagger$, one should work with pairs $H_{n,\dagger},H_{n,\ddagger}$ instead.
	\item  Finally, a lacking desirable property is that the result in \cite{FK06} is based on the assumption that $X_n$ are mapped into $X$. This leads to problems for example in the setting of hydrodynamic limits, see e.g. \cite{KL99}. In this context a Markov process is considered in which particles move around on a discrete lattice, e.g. $\bZ^d$. A typical state-space would be $X_n := \{0,1\}^{\bZ_d}$. After rescaling the lattice and speeding up time appropriately, the empirical measure associated to the particle locations converges to the solution of a diffusion equation, say in $X := L^1(\dd x)$. The convergence of measures to a profile in $L^1(\bR^d,\dd x)$, however, is considered with respect to the vague topology on $\cX := \cM(\bR^d)$.
	
	Thus, instead of considering spaces $X_n$ that get mapped into $X$ one wants to consider an auxiliary space $\cX$, in which both $X_n$ and $X$ get mapped. The convergence of elements is then considered as elements in $\cX$.
\end{itemize}

In this paper, we extend in Theorem \ref{theorem:convergence_of_resolvents} the Feng-Kurtz extension of the Barles-Perthame procedure to remedy these three issues. As a consequence, the Kurtz \cite{Ku73} convergence result gives us convergence of semigroups, see Theorem \ref{theorem:CL_extend}. 

In future work, the extended procedure will be used for a new proof of large deviations for Markov processes. In addition, in \cite{Kr19b} we give a framework to establish Gamma convergence of functionals on path-space.

\smallskip

As all the generalizations are quite technical, we start out in Section \ref{section:two_basic_results} with stating (without giving definitions of the required notions) a basic version of the convergence of viscosity solutions and the convergence of semigroups. This allows to quickly grasp the kind of results that are accessible. In this context, we will work with $X_n = \cX = X$ the notion of buc convergence and operators $H_n,H \subseteq C_b(X) \times C_b(X)$.

To set the stage for the more general results, we start in Section \ref{section:preliminaries} with some preliminaries that include a treatment of basic properties of our notion of convergence taking place on spaces $X_n,X,\cX$. All these results can be skipped on first reading assuming that $X_n = X = \cX$ and coincide with the ones of \cite{FK06} in the context that the $X_n$ are mapped into $X$ (in this case $X = \cX$).

We proceed in Section \ref{section:pseudo_resolvent_and_HJ} on a basic study of viscosity solutions for the Hamilton-Jacobi equation $f - \lambda H f = h$, as well as a study of pseudo-resolvents. To some extent these results are known in the literature, but as the results and proofs will be used as input for our main results later on, we collect these results for completeness. To summarize, we show that pseudo-resolvents can be used to construct viscosity solutions. On the other-hand, given well-posedness of the Hamilton-Jacobi equation, viscosity solutions can be used to construct a pseudo-resolvent. Finally, in this context, the pseudo-resolvent can be used to define a new Hamiltonian that satisfies the conditions for the semi-group generation result by \cite{CL71}.

In Sections \ref{section:convergence_of_operators_and_inverses} and \ref{section:convergence_of_semigroups} we proceed with our convergence statements, containing the two main Theorems \ref{theorem:convergence_of_resolvents} and \ref{theorem:CL_extend}.

Finally, we end in Section \ref{section:density_of_domain} with a short discussion on how to use the comparison principle for Hamilton-Jacobi equations to establish density of the domain of an operator constructed out of viscosity solutions. 

%
%

\section{Two basic convergence results} \label{section:two_basic_results}

To anticipate the general version of our two main results, we state in this section two simplified versions of these results. We will not give definitions of the required notions, as these will follow in more general context in Section \ref{section:preliminaries}. The notion of a pseudo-resolvent can be found as Definition \ref{defintion:pseudo_resolvent}.

We start with the convergence of viscosity solutions of Hamilton-Jacobi equations. A more general version is given as Theorem \ref{theorem:convergence_of_resolvents} below.

\begin{theorem} \label{theorem:main_theorem_basic_resolvent}
Suppose there are contractive pseudo-resolvents $R_n(\lambda) : C_b(X) \rightarrow C_b(X)$, $\lambda >0$ and operators $H, H_n \subseteq C_b(X) \times C_b(X)$, $n \geq 1$. Suppose in addition that
\begin{enumerate}[(a)]
	\item \label{item:convH_pseudoresolvents_solve_HJ_basic} For each $n \geq 1$, $\lambda > 0$ and $h \in C_b(X)$ the function $R_n(\lambda)h$ is a viscosity solution to $f - \lambda  H_n f= h$. 
	\item \label{item:convH_strict_equicont_resolvents_basic} We have local strict equi-continuity on bounded sets: for all compact sets $K \subseteq X$, $\delta > 0$ and $\lambda_0 > 0$, there is a compact set $\hat{K} = \hat{K}(K,\delta,\lambda_0)$ such that for all $n$ and $h_{1},h_{2} \in C_b(X)$ and $0 < \lambda \leq \lambda_0$ we have
	\begin{multline*}
	\sup_{y \in K} \left\{ R_n(\lambda)h_{1}(y) - R_n(\lambda)h_{2}(y) \right\} \\
	\leq \delta \sup_{x \in X} \left\{ h_{1}(x) - h_{2}(x) \right\} + \sup_{y \in \hat{K}} \left\{ h_{1}(y) - h_{2}(y) \right\}.
	\end{multline*}
	\item For each $(f,g) \in H$ there are $(f_n,g_n) \in H_n$ such that $buc-\lim f_n = f$ and $buc-\lim g_n = g$.
	\item There is a buc-dense set $D \subseteq C_b(X)$ such that the comparison principle holds for the Hamilton-Jacobi equation $f - \lambda H f = h$ for all $h \in D$ and $\lambda > 0$.
\end{enumerate}
Then there are operators $R(\lambda) : C_b(X) \rightarrow C_b(X)$ for $\lambda > 0$ that are locally strictly equi-continuous on bounded sets: for all compact sets $K \subseteq X$, $\delta > 0$ and $\lambda_0 > 0$, there is a compact set $\hat{K} = \hat{K}(K,\delta,\lambda_0)$ such that for all $h_{1},h_{2} \in C_b(X)$ and $0 < \lambda \leq \lambda_0$ we have
\begin{equation*}
\sup_{y \in K} \left\{ R(\lambda)h_{1}(y) - R(\lambda)h_{2}(y) \right\} \leq \delta \sup_{x \in X} \left\{ h_{1}(x) - h_{2}(x) \right\} + \sup_{y \in \hat{K}} \left\{ h_{1}(y) - h_{2}(y) \right\}.
\end{equation*}
For each $\lambda > 0$ and $h \in D$ the function $R(\lambda)h$ is the unique viscosity solution to $f - \lambda Hf = h$. In addition if $buc-\lim h_n = h$ then $buc-\lim R_n(\lambda)h_n = R(\lambda) h$.
\end{theorem}

The next result uses the convergence of the pseudo-resolvents to obtain the convergence of semigroups. The key ingredient in this context are the semigroup generation and convergence results of \cite{CL71,Ku73}. A more general version follows in Theorem \ref{theorem:CL_extend}.

\begin{theorem} \label{theorem:main_theorem_basic_semigroup}
	Suppose we are in the setting of Theorem \ref{theorem:main_theorem_basic_resolvent}. In addition suppose that there are operators $V_n(t) : C_b(X) \rightarrow C_b(X)$, $t \geq 0$ that form a semigroup. Suppose that
	\begin{enumerate}[(a)]
		\item there is a buc-dense subset $\cD_{n,0}$ such that for every $n \geq 1$, $f \in \cD_{n,0}$ and $x \in X$:
		\begin{equation*}
		\lim_{m \rightarrow \infty} R_n\left(\frac{t}{m}\right)^mf(x) = V_n(t)f(x).
		\end{equation*}
		\item \label{item:convH_strict_equicont_semigroups_basic}  We have local strict equi-continuity on bounded sets for the semigroups: for all compact sets $K \subseteq X$, $\delta > 0$ and $t_0 > 0$, there is a compact set $\hat{K} = \hat{K}(K,\delta,\lambda_0)$ such that for all $n$ and $h_{1},h_{2} \in C_b(X)$ and $0 \leq t \leq t_0$ that
		\begin{multline*}
		\sup_{y \in K} \left\{ V_n(t)h_{1}(y) - V_n(t)h_{2}(y) \right\} \\
		\leq \delta \sup_{x \in X} \left\{ h_{1}(x) - h_{2}(x) \right\} + \sup_{y \in \hat{K}} \left\{ h_{1}(y) - h_{2}(y) \right\}.
		\end{multline*}
		\item We have $buc-\lim_{\lambda \downarrow 0} R(\lambda)h = h$ and for each $n$ we have $buc-\lim_{\lambda \downarrow 0} R_n(\lambda)h = h$.
	\end{enumerate}
Then there are operators $V(t) : C_b(X) \rightarrow C_b(X)$ for $t \geq 0$ that are locally strictly equi-continuous on bounded sets: for all compact sets $K \subseteq X$, $\delta > 0$ and $t_0 > 0$, there is a compact set $\hat{K} = \hat{K}(K,\delta,t_0)$ such that for all $h_{1},h_{2} \in C_b(X)$ and $0 \leq t \leq t_0$ we have
\begin{equation*}
\sup_{y \in K} \left\{ V(t)h_{1}(y) - V(t)h_{2}(y) \right\} \leq \delta \sup_{x \in X} \left\{ h_{1}(x) - h_{2}(x) \right\} + \sup_{y \in \hat{K}} \left\{ h_{1}(y) - h_{2}(y) \right\}.
\end{equation*}
In addition there are subsets $\cD_n,\cD \subseteq C_b(X)$ that are buc-dense such that if $f_n \in \cD_n$, $f \in \cD$ such that $buc-\lim f_n = f$ and $t_n \rightarrow t$, then $buc-\lim V_n(t_n)f_n = V(t)f$.
\end{theorem}

Both results will foll as special cases of much more general result that we will prove in the following sections.

\section{Preliminaries} \label{section:preliminaries}

\subsection{Basic definitions}

All spaces in this paper are assumed to be completely regular spaces. Let $X$ be a space then we denote by $C_b(X)$ the set of continuous and bounded functions into $\bR$. We denote by $\Ba(X)$ the space of Baire measurable sets (the $\sigma$-algebra generated by $C_b(X)$.) By $M(X)$, we denote by set of Baire measurable functions from $X$ into $\overline{\bR} := [-\infty,\infty]$. $M_b(X)$ denotes the set of bounded Baire measurable functions.  Denote
\begin{align*}
USC_u(X) & := \left\{f \in M(X) \, \middle| \,  f \, \text{upper semi-continuous}, \sup_x f(x) < \infty \right\}, \\
LSC_l(X) & := \left\{f \in M(X) \, \middle| \, f \, \text{lower semi-continuous}, \inf_x f(x) > \infty \right\}.
\end{align*}
For $g \in M(X)$ denote by $g^*,g_* \in M(X)$ the upper and lower semi-continuous regularizations of $g$.

\subsection{Viscosity solutions}

Let $X$ and $Y$ be two spaces. Let $\gamma : Y \rightarrow X$ be continuous and surjective.

We consider operators $A \subseteq M(X) \times C(Y)$. If $A$ is single valued and $(f,g) \in A$, we write $Af := g$. We denote $\cD(A)$ for the domain of $A$ and $\cR(A)$ for the range of $A$.

\begin{definition}
	Let $A_\dagger \subseteq LSC_l(X) \times USC_u(Y)$ and $A_\ddagger \subseteq USC_u(X) \times LSC_l(Y)$. Fix $h_1,h_2 \in M(X)$. Consider the equations
	\begin{align} 
	f -  A_\dagger f & = h_1, \label{eqn:differential_equation_Adagger} \\
	f - A_\ddagger f & = h_2. \label{eqn:differential_equation_Addagger}
	\end{align}
	\begin{description}
	\item[Classical solutions] We say that $u$ is a \textit{classical subsolution} of equation \eqref{eqn:differential_equation_Adagger} if there is a $g$ such that $(u,g) \in A_\dagger$ and $u - g \leq h$. We say that $v$ is a \textit{classical supersolution} of equation \eqref{eqn:differential_equation_Addagger} if there is $g$ such that $(v,g) \in A_\ddagger$ and $v - g \geq h$. We say that $u$ is a \textit{classical solution} if it is both a sub- and a supersolution.
	\item[Viscosity subsolutions] We say that $u : X \rightarrow \bR$ is a \textit{subsolution} of equation \eqref{eqn:differential_equation_Adagger} if $u \in USC_u(X)$ and if, for all $(f,g) \in A_\dagger$ such that $\sup_x u(x) - f(x) < \infty$ there is a sequence $y_n \in Y$ such that
	\begin{equation} \label{eqn:subsol_optimizing_sequence}
	\lim_{n \rightarrow \infty} u(\gamma(y_n)) - f(\gamma(y_n))  = \sup_x u(x) - f(x),
	\end{equation}
	and
	\begin{equation} \label{eqn:subsol_sequence_outcome}
	\limsup_{n \rightarrow \infty} u(\gamma(y_n)) - g(y_n) - h_1(\gamma(y_n)) \leq 0.
	\end{equation}
	\item[Viscosity supersolution] 	We say that $v : X \rightarrow \bR$ is a \textit{supersolution} of equation \eqref{eqn:differential_equation_Addagger} if $v \in LSC_l(X)$ and if, for all $(f,g) \in A_\ddagger$ such that $\inf_x v(x) - f(x) > - \infty$ there is a sequence $y_n \in Y$ such that
	\begin{equation} \label{eqn:supersol_optimizing_sequence}
	\lim_{n \rightarrow \infty} v(\gamma(y_n)) - f(\gamma(y_n))  = \inf_x v(x) - f(x),
	\end{equation}
	and
	\begin{equation} \label{eqn:supersol_sequence_outcome}
	\liminf_{n \rightarrow \infty} v(\gamma(y_n)) - g(y_n) - h_2(\gamma(y_n)) \geq 0.
	\end{equation}
	\item[Viscosity solution] We say that $u$ is a \textit{solution} of the pair of equations \eqref{eqn:differential_equation_Adagger} and \eqref{eqn:differential_equation_Addagger} if it is both a subsolution for $A_\dagger$ and a supersolution for $A_\ddagger$.
	\item[Comparison principle] We say that  \eqref{eqn:differential_equation_Adagger} and \eqref{eqn:differential_equation_Addagger} satisfy the \textit{comparison principle} if for every subsolution $u$ to \eqref{eqn:differential_equation_Adagger} and supersolution $v$ to \eqref{eqn:differential_equation_Addagger}, we have
	\begin{equation} \label{eqn:comparison_estimate}
	\sup_x u(x) - v(x) \leq \sup_x h_1(x) - h_2(x).
	\end{equation}
	If $H = A_\dagger = A_\ddagger$, we will say that the comparison principle holds for $f - \lambda Af = h$, if for any subsolution $u$ for $f - \lambda Af = h_1$ and supersolution $v$ of $f - \lambda Af = h_2$ the estimate in \eqref{eqn:comparison_estimate} holds. 
	\end{description}
\end{definition}

Often, as in Section \ref{section:pseudo_resolvent_and_HJ} below, $Y = X$ and $\gamma(x) = x$ simplifying the definitions above.

\subsection{Operators}

For an operator $A \subseteq M(X) \times M(Y)$ and $c \geq 0$ we write $cA \subseteq M(X) \times M(Y)$ for the operator
\begin{equation*}
c \cdot A := \left\{ (f,c\cdot g) \, \middle| \, (f,g) \in A \right\}.
\end{equation*}
Here we write $c \cdot g$ for the function
\begin{equation*}
c \cdot g(x) := \begin{cases}
cg(x) & \text{if } g(x) \in \bR, \\
\infty & \text{if } g(x) = \infty, \\
- \infty & \text{if } g(x) = - \infty.
\end{cases}
\end{equation*}

The next set of properties is mainly relevant in the setting that $Y = X$.

\begin{definition}
	
	\begin{description}
		\item[Contractivity] We say that $T \subseteq M(X) \times M(X)$ is \textit{contractive} if for all $f_1,f_2 \in \cD(T)$:
		\begin{align*}
		\sup_x Tf_1(x) - Tf_2(x) \leq \sup_x f_1(x) - f_2(x), \\
		\inf_x Tf_1(x) - Tf_2(x) \geq \inf_x f_1(x) - f_2(x).
		\end{align*}
		If in addition $T0 = 0$, contractivity implies that $\sup_x Tf(x) \leq \sup_x f(x)$ and $\inf_x Tf(x) \geq \inf_x f(x)$.
		\item[Dissipativity] We say $A \subseteq M(X) \times M(X)$ is \textit{dissipative} if for all $(f_1,g_1),(f_2,g_2) \in A$ and $\lambda > 0$ we have
		\begin{equation*}
		\vn{f_1 - \lambda g_1 - (f_2 - \lambda g_2)} \geq \vn{f_1 - f_2};
		\end{equation*}
		\item[The range condition] We say $A \subseteq M(X) \times M(X)$ satisfies \textit{the range condition} if for all $\lambda > 0$ we have: the uniform closure of $\cD(A)$ is a subset of $\cR(\bONE - \lambda A)$.
	\end{description}

\end{definition}

The following theorem was proven for accretive operators but can be easily translated into dissipative operators by changing $A$ by $-A$.

\begin{theorem}[Crandall-Liggett \cite{CL71}] \label{theorem:CL}
	Let $A$ be an operator on a Banach space $E$. Suppose that 
	\begin{enumerate}[(a)]
		\item $A$ is dissipative,
		\item $A$ satisfies the range condition.
	\end{enumerate}
	Denote by $R(\lambda,A) = \left(\bONE- \lambda A\right)^{-1}$. Then there is a strongly continuous (for the norm) contraction semigroup $S(t)$ defined on the uniform closure of $\cD(A)$ and for all $t \geq 0$ and $f$ in the uniform closure of $\cD(A)$
	\begin{equation*}
	\lim_n \vn{R\left(\tfrac{t}{n},A\right)^n f - S(t) f} = 0.
	\end{equation*}
\end{theorem}

%
%

\subsection{Operators and the strict and buc topology}

In addition to normed spaces, we consider bounded and uniform convergence on compacts (buc-convergence). This notion of convergence for functions on $C_b(X)$ is more natural from an applications point of view. This is due to the fact that it is the restriction of the locally convex strict topology restricted to sequences, see e.g. \cite{Bu58,Se72}. Indeed, it is the strict topology for which most well known results generalize (under appropriate conditions on the topology, e.g. $X$ Polish): Stone-Weierstrass, Arzelà-Ascoli and the Riesz representation theorem. We define both notions.

\begin{definition}[buc convergence] \label{definition:buc_convergence}
	Let $f_n \in C_b(X)$ and $f \in C_b(X)$. We say that $f_n$ converges \textit{bounded and uniformly on compacts} (buc) if $\sup_n \vn{f_n} < \infty$ and if for all compact $K \subseteq X$:
	\begin{equation} \label{eqn:buc_uniform}
	\lim_n \sup_{x \in K} \left|f_n(x) - f(x)\right| = 0.
	\end{equation}
\end{definition}

Note \eqref{eqn:buc_uniform} can be replaced by $f_n(x_n) \rightarrow f(x)$ for all sequences $x_n \in K$ that converge to $x \in K$.

\begin{definition}
	The \textit{(sub) strict topology} $\beta$ on the space $C_b(X)$ for a completely regular space $X$ is generated by the collection of semi-norms
	\begin{equation*}
	p(f) := \sup_n a_n \sup_{x \in K_n} |f(x)|
	\end{equation*}
	where $K_n$ are compact sets in $X$ and where $a_n \geq 0$ and $a_n \rightarrow 0$.
\end{definition}

	\begin{remark} \label{remark:strict_implies_buc}
		The (sub)strict topology is the finest locally convex topology that coincides with the compact open topology on bounded sets. Thus, a sequence converges strictly if and only if it converges buc.
		
		In the literature on locally convex spaces, the strict topology is usually referred to as the substrict topology, but on Polish spaces, among others, these topologies coincide, see \cite{Se72}.
	\end{remark}

	\begin{definition}
		\begin{enumerate}[(a)]
			\item Denote $B_r := \left\{f \in C_b(X) \, \middle| \, \vn{f} \leq r \right\}$. We say that a set $D$ is \textit{quasi-closed} if for all $r \geq 0$ the set $D \cap B_r$ is closed for the strict topology (or equivalently for the compact open or buc topologies).
			\item We say that $\widehat{D}$ is the \textit{quasi-closure} of $D$ if $\widehat{D} = \bigcup_{r > 0} \widehat{D}_r$, where $\widehat{D}_r$ is the strict closure of $D \cap B_r$.
			\item We say that $D_1$ is \textit{quasi-dense} in $D_2$ if $D_1 \cap B_r$ is strictly dense in $D_2 \cap B_r$ for all $r \geq 0$.
		\end{enumerate}
	\end{definition}

	Next, we consider operators with respect to a hierarchy of statements regarding continuity involving the strict topology. The proof can be found in Appendix \ref{appendix:proof_of_topology_implications}.

\begin{proposition} \label{proposition:buc_cont_equiv_to_double_bound}
	Let $T : C_b(X) \rightarrow C_b(X)$. Consider
	\begin{enumerate}[(a)]
		\item $T$ is strictly continuous.
		\item \label{item:continuity_strict_intermediate} For all $\delta > 0$, $r > 0$, and compact sets $K$ there are $C_0(r)$, $C_1(\delta,r)$ and a compact set $\hat{K}(K,\delta,r)$ such that 
		\begin{equation*}
		\sup_{x \in K} |Tf(x) - Tg(x)| \leq \delta C_0(r) + C_1(\delta, r) \sup_{x \in \hat{K}(K,\delta,r)} |f(x) - g(x)|
		\end{equation*}
		for all $f,g \in C_b(X)$ such that $\vn{f} \vee \vn{g} \leq r$.
		\item $T$ is strictly continuous on bounded sets.
	\end{enumerate}
Then (a) implies (b) and (b) implies (c).
\end{proposition}

\begin{remark}
	There is not much room between properties (a) and (c). In the case that $X$ is Polish space, and $T$ is linear then (a) and (c) are equivalent, see e.g. \cite[Corollary 3.2 and Theorem 9.1]{Se72}. It is unclear to the author whether (b) and (c) are equivalent in general.
\end{remark}

At various points in the paper, we will work with operators that are constructed by taking closures on dense sets. To do so, we need continuity properties. Even though working with (a) of \ref{proposition:buc_cont_equiv_to_double_bound} would be the desirable from a functional analytic point of view,  \eqref{item:continuity_strict_intermediate} is much more explicit, and also suffices for our analysis.

The following result is proven in \cite[Lemma A.11]{FK06}.

\begin{lemma} \label{lemma:extension_operator}
	Suppose that an operator $T : D \subseteq C_b(X) \rightarrow C_b(X)$ satisfies (b) of Proposition \ref{proposition:buc_cont_equiv_to_double_bound}.  Then $T$ has an extension to the quasi-closure $\widehat{D}$ of $D$ that also satisfies property \eqref{item:continuity_strict_intermediate} of Proposition \ref{proposition:buc_cont_equiv_to_double_bound}  (with the same choice of $\hat{K}$).
\end{lemma}

\subsection{A general setting of convergence of spaces} \label{section:general_convergence1}

In previous section, we have studied buc convergence and the strict topology. This suffices for convergence problems in the context where all Hamilton-Jacobi equations $f - \lambda H_n f = h$ and $f - \lambda Hf = h$ are based on the same space $X$. In practice, however, one runs into situations where this is not natural. In the context of simple slow-fast systems for example, one typically works with $X_n = E \times F$ and $X = E$. That is, we have a slow system on $E$ that depends on a fast system taking values on $F$. Taking limits, we end up with a slow system on $E$ with coefficients that are suitable averages over $F$. Thus, we need to connect $X_n$ to $X$ via a mapping $\eta_n$ (e.g. a projection on the first coordinate) and extend the notion of buc convergence to allow for functions $f_n$ on $X_n$ to converge to $X$.

\smallskip

To extend the notion of buc convergence, we need to decide what `uniform convergence' on compacts means. Following Definition \ref{definition:buc_convergence}, we saw that $f_n$ converges to $f$ buc if $\sup_n \vn{f_n} < \infty$ and if for all compact sets $K$ we have $f_n(x_n) \rightarrow f(x)$ for all sequences $x_n \in K$ converging to $x \in K$.

In the context of distinct $X_n$ and $X$, there is no natural analogue of the compact set $K$. Instead, we will work a  sequence of compact sets. Namely, we will choose compact sets $K_n \subseteq X_n$ that `converge' to a compact set $K \subseteq X$. Then $f_n$ converges to $f$ if $\sup_n \vn{f_n} < \infty$ and if for each of these sequences of compact sets and $x_n \in K_n$ converging to $x \in K$, we have $f_n(x_n) \rightarrow f(x)$.

\smallskip

We turn to the rigorous definition. We will slightly extend our discussion above by allowing spaces $X_n$ and $X$ such that the $X_n$ are not naturally embedded in $X$. Instead, we will map all spaces to a common space $\cX$ in which `$X_n$ converges to $X$'. 

\begin{assumption} \label{assumption:abstract_spaces}
	Consider spaces $X_n$ and $X$, some space $\cX$, Baire measurable maps $\eta_n : X_n \rightarrow \cX$ and a Baire measurable injective map $\eta : X \rightarrow \cX$. 
\end{assumption}

\begin{definition}[Kuratowski convergence]
	Let $\{O_n\}_{n \geq 1}$ be a sequence of subsets in a space $\cX$. We define the \textit{limit superior} and \textit{limit inferior} of the sequence as
	\begin{align*}
	\limsup_{n \rightarrow \infty} O_n & := \left\{x \in \cX \, \middle| \, \forall \, U \in \cU_x \, \forall \, N \geq 1 \, \exists \, n \geq N: \, O_n \cap U \neq \emptyset \right\}, \\
	\liminf_{n \rightarrow \infty} O_n & := \left\{x \in \cX \, \middle| \, \forall \, U \in \cU_x \, \exists \, N \geq 1 \, \forall \, n \geq N: \, O_n \cap U \neq \emptyset \right\}.
	\end{align*}
	where $\cU_x$ is the collection of open neighbourhoods of $x$ in $\cX$. If $O := \limsup_n O_n = \liminf_n O_n$, we write $O = \lim_n O_n$ and say that $O$ is the Kuratowski limit of the sequence $\{O_n\}_{n \geq 1}$.
\end{definition}

\begin{assumption} \label{assumption:abstract_spaces_q}
	There is a directed set $\cQ$ (partially ordered set such that every two elements have an upper bound). For each $q \in \cQ$, we have compact sets $K_n^q \subseteq X_n$ a compact set $K^q \subseteq X$ such that
	\begin{enumerate}[(a)]
		\item If $q_1 \leq q_2$, we have $K^{q_1} \subseteq K^{q_2}$ and for all $n$ we have $K_{n}^{q_1} \subseteq K_n^{q_2}$.
		\item For all $q \in \cQ$ and each sequence $x_n \in K_n^q$, every subsequence of $x_n$ has a further subsequence that is converging to a limit $x \in K^q$ (that is: $\eta_n(x_n) \rightarrow \eta(x)$ in $\cX$). 
		\item For each compact set $K \subseteq X$, there is a $q \in \cQ$ such that
		\begin{equation*}
		\eta(K) \subseteq \liminf_n \eta_n(K_n^q).
		\end{equation*}
	\end{enumerate}
\end{assumption}

\begin{remark}
	Note that (b) implies that $\limsup_n \eta_n(K_n^q) \subseteq \eta(K^q)$. Note that (b) follows if  $\bigcup_n \eta_n(K_n^q)$ is a subset of $\eta(K^q)$ and the topology on $K^q$ is metrizable.
\end{remark}

Conditions (b) should be interpreted in the sense that $K^q$ is larger than the `limit' of the sequence $K_n$, whereas (c) should be interpreted in the sense that each compact $K$ in $X$ is contained in a limit of that type.

We will say that a sequence $x_n \in X_n$ converges to $x \in X$ in the sense that $\eta_n(x_n) \rightarrow \eta(x)$ in $\cX$. Dual to the notion of convergence in a topological space, there is the notion of convergence of functions.

\begin{definition} \label{definition:abstract_LIM}
	Let Assumptions \ref{assumption:abstract_spaces} and \ref{assumption:abstract_spaces_q} be satisfied. For each $n$ let $f_n \in M_b(X_n)$ and $f \in M_b(X)$. We say that $\LIM f_n = f$ if
	\begin{itemize}
		\item $\sup_n  \vn{f_n} < \infty$,
		\item if for all $q \in \cQ$ and $x_n \in K_n^q$ converging to $x \in K^q$ we have
		\begin{equation*}
		\lim_{n \rightarrow \infty} \left|f_n(x_n) - f(x)\right| = 0.
		\end{equation*}
	\end{itemize}
\end{definition}

The notion of bounded and uniform on compacts (buc) is  the prime example of a notion of $\LIM$. For a second example see Example 2.7 in \cite{FK06}. 

\begin{example}[buc convergence] \label{example:LIM_buc_convergence}
	Consider some space $X$ in which all compact sets are metrizable, and suppose that $X_n = X$ and that $\eta_n$ is the identity map for all $n$. In this context, we can choose $\cX = X$ and $\eta$ the identity map. $\bQ$ is the set of compact subsets. For $\cK \in \cQ$, we take $K_n^\cK = K^\cK = \cK$.
	
	Note that we need metrizable compact sets to extract converging subsequences for Assumption \ref{assumption:abstract_spaces_q} (b).
	
	\smallskip

	We have $\LIM f_n = f$ if and only if $\sup_{n} \vn{f_n} < \infty$ and if for all $\cK$ and all sequences $x_n \in \cK$ converging to $x \in \cK$, we have $\lim_n f_n(x_n) = f(x)$.
\end{example}

\begin{remark}
	In the setting that $\cX = X$ whose topologies coincide, we can compare the notion of $\LIM$ we introduced to that which is used in \cite{FK06}. Indeed, it is straightforward to show that both notions of $\LIM f_n = f$ for a sequence of functions coincide if the limiting function $f$ is continuous. 
\end{remark}

\begin{remark}
	The notion of $\LIM$ is subtle. It does not require $f_n(x_n) \rightarrow f(x)$ for all sequences $x_n$ such that $\eta_n(x_n) \rightarrow \eta(x)$ in $\cX$.
	
	For example, let $X_n = \bR\times \bR$, $\cX = X = \bR$, $\eta_n(x,y) = x$ and $\eta(x) = x$. We could work with an index set $\cQ$ consisting of all compact sets $[a,b]\times[c,d]$ in $\bR^2$. Then $K_n^{[a,b]\times[c,d]} = [a,b]\times[c,d]$ and $K^{[a,b]\times[c,d]} = [a,b]$. Clearly the sequence $x_n := (x,n)$ satisfies $\eta_n(x,n) = x$ which converges to $x$. There is however, no compact set $[a,b]\times [c,d]$ such that $(x,n)$ lies in this set for all $n$. Thus, we do not need to check convergence along this sequence in Definition \ref{definition:abstract_LIM}.
\end{remark}

\begin{remark} \label{remark:balance_LIM_and_equi_cont}
	Proceeding with last remark. Note that we could have chosen different compact sets with the same index set. E.g., we could have chosen $K_n^{[a,b]\times[c,d]} = [a,b]\times[nc,nd]$ and $K^{[a,b]\times[c,d]} = [a,b]$. This leads to a larger collection of sequences for which we have to verify convergence for $\LIM$.
	
	In Section \ref{section:equi_cont_and_LIM} below, we will see that we can define a notion of equi-continuity of operators on the spaces $X_n$ based on the set $\cQ$ and compacts $K_n^q$.
	
	Indeed, in Condition \ref{condition:convergence_of_generators_and_conditions_extended_supsuperlim}, key for our main results, we will assume that we have converge of Hamiltonians in the sense of $\LIM$, and have equi-continuity for the resolvents in terms of $K_n^q$. This leads to a careful balance: choose small sets $K_n^q$, then verifying convergence with $\LIM$ is easy whereas verifying equi-continuity becomes hard and vice versa. Thus, the choice of $K_n^q$ is context dependent and requires insight into the problem at hand.
\end{remark}

The characterization of $f = \LIM f_n$ allows for generalization of the $\limsup$ and $\liminf$ as well.

\begin{definition} \label{definition:abstract_LIMSUP_LIMINF}
	Let Assumptions \ref{assumption:abstract_spaces} and \ref{assumption:abstract_spaces_q} be satisfied. Let $f_n \in M(X_n)$. 
	\begin{enumerate}[(a)]
		\item \label{item:abstract_LIMSUP} Let $f \in USC_u(X)$. We say that $\LIMSUP f_n = f$ if
		\begin{itemize}
			\item $\sup_{n} \sup_{x \in X_n} f_n(x) < \infty$,
			\item if 
			\begin{equation*}
			f(x) = \sup_{q \in \cQ} \sup \left\{ \limsup_{n \rightarrow \infty} f_n(x_n) \, \middle| \, x_n \in K_n^q, \, \eta_n(x_n) \rightarrow \eta(x) \right\}.
			\end{equation*}
		\end{itemize}
		\item Let $f \in LSC_l(X)$. We say that $\LIMINF_n f_n = f$ if
		\begin{itemize}
			\item $\inf_{n}  \inf_{x \in X_n} f_n(x) > - \infty$,
			\item if 
			\begin{equation*}
			f(x) := \inf_{q \in \cQ} \inf \left\{ \liminf_{n \rightarrow \infty} f_n(x_n) \, \middle| \, x_n \in K_n^q, \, \eta_n(x_n) \rightarrow \eta(x) \right\}.
			\end{equation*}
		\end{itemize}
	\end{enumerate} 
\end{definition}

The following is immediate.

\begin{lemma}
	Let Assumptions \ref{assumption:abstract_spaces} and \ref{assumption:abstract_spaces_q} be satisfied. Suppose that $\LIMSUP_n f_n \leq f \leq \LIMINF f_n$, then $\LIM f_n = f$.
\end{lemma}

\subsection{Joint equicontinuity of operators and \texorpdfstring{$\LIM$}{LIM}} \label{section:equi_cont_and_LIM}

It is a general fact from topology that if $T_n$ are equi-continuous functions on some space and if $f_{1,n} \rightarrow f$ and $f_{2,n} \rightarrow f$, then $T_nf_{1,n} - Tf_{2,n} \rightarrow 0$. 

We now show that equi-continuity in the sense of \eqref{item:continuity_strict_intermediate} of Proposition \ref{proposition:buc_cont_equiv_to_double_bound} combines with the notion of $\LIM$ in a similar way. Afterwards, we will show that we can use $\LIM$ and a collection of equi-continuous operators to define a limiting operator.

\begin{definition} \label{definition:joint_strict_equicontinuity_boundedsets}
	Let Assumptions \ref{assumption:abstract_spaces} and \ref{assumption:abstract_spaces_q} be satisfied. Let $T_n : B_n \subseteq M_b(X_n) \rightarrow B_n$ be operators. 
	
	We say that the collection $\{T_n\}_{n \geq 1}$ is \textit{strictly equi-continuous on bounded sets} if the following holds.
	For all $q \in \cQ$, $r > 0$ and $\delta > 0$, there is a $\hat{q} \in \cQ$ and constants $C_0(r), C_1(\delta,r)$ such that for all $n$ and $h_{1},h_{2} \in B_n$ with $\vn{h_1} \vee \vn{h_2} \leq r$ we have
	\begin{equation*}
	\sup_{y \in K_n^q} \left\{ T_n h_{1}(y) - T_n h_{2}(y) \right\} \leq \delta C_0(r) + C_1(\delta,r) \sup_{y \in K^{\hat{q}}_n} \left\{ h_{1}(y) - h_{2}(y) \right\}.
	\end{equation*}
\end{definition}

\begin{lemma} \label{lemma:joint_equicontinuity_with_convergence}
	Let Assumptions \ref{assumption:abstract_spaces} and \ref{assumption:abstract_spaces_q} be satisfied. Let $T_n : B_n \subseteq M_b(X_n) \rightarrow B_n$ be a collection of operators that is strictly equi-continuous on bounded sets.
	
	Suppose that $h_{1,n},h_{2,n} \in B_n$ and that $\LIM h_{1,n} = \LIM h_{2,n}$. Then it holds that $\LIM T_nh_{1,n} - T_n h_{2,n} = 0$.
	In particular, if $\LIM T_n h_{1,n}$ exists, then $\LIM T_n h_{2,n}$ exists also and is the same.
\end{lemma}

\begin{proof}
	Pick $h_{1,n},h_{2,n} \in B_n$ and that $\LIM h_{1,n} = \LIM h_{2,n}$. We establish that $\LIMSUP T_n h_{1,n} - T_n h_{2,n} \leq 0$. By interchanging the roles of $h_{1,n}$ and $h_{2,n}$ this yields the statement for $\LIMINF$ which establishes the claim.
	
	To do so, it suffices for any $q \in \cQ$ and $x_n \in K_n^q$ and $x$ such that $\eta_n(x_n) \rightarrow \eta(x)$ to establish that
	\begin{equation*}
	\limsup_n T_n h_{1,n}(x_n) - T_n h_{2,n}(x_n) \leq 0.
	\end{equation*}
	As $\LIM h_{1,n}, \LIM h_{2,n}$ exists, there is some $r > 0$ such that $\sup_n \vn{h_{1,n}} \vee \vn{h_{2,n}} \leq r$. Thus, by joint strict local equi-continuity of the operators $\{T_n\}$ we can find for any $\delta >0$ a $\hat{q}$ and constants $C_0(\delta), C_1(\delta,r)$ such that
	\begin{multline*}
	\limsup_n T_n h_{1,n}(x_n) - T_n h_{2,n}(x_n) \\
	\leq \delta C_0(r) + C_{1}(r,\delta) \limsup_n \sup_{y \in K_n^{\hat{q}}} h_{1,n}(y) - h_{2,n}(y).
	\end{multline*}
	As $\LIM h_{1,n} = \LIM h_{2,n}$ the $\limsup_n$ on the right equals $0$. Sending $\delta \rightarrow 0$, the first claim follows.
	The final claim is a direct consequence of the triangle inequality.
\end{proof}

In next proposition, we show how to use the result of previous lemma to construct a limiting operator our of a sequence of operators that are strictly equi-continuous on bounded sets.
	
	\begin{proposition} \label{proposition:extending_operators_via_LIM}
		Let Assumptions \ref{assumption:abstract_spaces} and \ref{assumption:abstract_spaces_q} be satisfied. Let $T_n : B_n \subseteq M_b(X_n) \rightarrow B_n$ be a strictly equi-continuous on bounded sets.
		Suppose the spaces $B_n$ are such that there is a $M > 0$ such that for all $h \in C_b(X)$ there are $h_n \in B_n$ such that $\LIM h_n = h$ and $\sup_n \vn{h_n} \leq M \vn{h}$.

		Set 
		\begin{equation*}
		\cD(T) := \left\{h \in C_b(X) \, \middle| \, \exists \, h_n \in B_n: h = \LIM h_n, \LIM T_nh_n \text{exists and is continuous}  \right\}
		\end{equation*}
		and $Th = \LIM T_nh_n$. Note that $T$ is well defined because of Lemma \ref{lemma:joint_equicontinuity_with_convergence}
		
		Then:
		\begin{enumerate}[(a)]
			\item $T$ is strictly continuous on bounded sets in the sense of \eqref{item:continuity_strict_intermediate} of Proposition \ref{proposition:buc_cont_equiv_to_double_bound}.
			\item The set $\cD(T)$ is quasi-closed in $C_b(X)$.
			\item If $h \in \cD(T)$ and $h_n \in B_n$ such that $\LIM h_n = h$, then $\LIM T_n h_n = Th$.
		\end{enumerate}
	\end{proposition}

	\begin{remark}
		In Lemma \ref{lemma:construction_f_n_that_converge_to_f} below, we will see that in the context that if the maps $\eta_n$ are continuous and $\eta$ is a homeomorphism onto its image, we can indeed always find $h_n$ such that $\LIM h_n = h$ and $\sup_n \vn{h_n} \leq \vn{h}$.
	\end{remark}

	The proof is inspired by Lemma 7.16 (b) and (c) in \cite{FK06}.
	
	\begin{proof}[Proof of Proposition \ref{proposition:extending_operators_via_LIM}]
		We start by by proving (a). Fix $r > 0$, a compact set $K \subseteq X$ and $\delta > 0$. We prove that there are constants $\hat{C}_0(r), \hat{C}_1(\delta,r)$ and a compact set $\widehat{K} = \widehat{K}(K,\delta,r)$ such that
		\begin{equation*}
		\sup_{x \in K} |Tf(x) -Tg(x)| \leq \delta \hat{C}_0(r) + \hat{C}_1(\delta, r) \sup_{x \in \hat{K}} |f(x) - g(x)|
		\end{equation*}
		for all $f,g \in \cD(T)$ such that $\vn{f} \vee \vn{g} \leq r$.
		
		\smallskip
		
		Thus, fix $f,g \in \cD(T)$ such that $\vn{f}\vee \vn{g} \leq r$ and let $x_0 \in K$ be such that 
		\begin{equation*}
		Tf(x_0) - Tg(x_0) = \sup_{y \in K} Tf(y) - Tg(y)
		\end{equation*}
		Let $f_n,g_n \in B_n$ such that $\LIM f_n = f$ and $\LIM g_n = g$ and $\sup_n \vn{f_n} \leq M \vn{f}$, $\sup_n \vn{g_n} \leq M \vn{g}$.  By Assumption \ref{assumption:abstract_spaces_q} (c), there is a $q$ with $K \subseteq \liminf_n \eta_n(K_n^q)$. Because of this, we can choose $x_n \in K_n^q$ such that $\lim_n \eta_n(x_n) = \eta(x_0)$ in $\cX$. By definition of $\cD(T)$ and $T$, we have $Tf(x_0) - Tg(x_0) = \lim_n T_nf_{n}(x_n) - T_n g_{n}(x_n)$. It follows by strict equi-continuity on bounded sets that there is a $\hat{q}$, $C_0(r)$ and $C_1(r,\delta)$ such that
		\begin{align*}
		\sup_{y \in K} Tf(y) - Tg(y) & = \lim_n T_nf_{n}(x_n) - T_ng_{n}(x_n) \\
		& \leq \lim_n \left(\delta C_0(r) M + C_1(r,\delta) \sup_{y \in K_n^{\widehat{q}}} \left\{ f_{n}(y) - g_{n}(y) \right\}\right).
		\end{align*}
		Set $\hat{C}_0(r) = MC_0(r)$ and $\hat{C}_1(\delta,r) = C_1(\delta,r)$. Finally, let $y_n \in K_n^{\widehat{q}}$ be such that $f_n(y_n) - g_n(y_n) + n^{-1} \geq  \sup_{y \in K_n^{\widehat{q}}} \left\{ f_{n}(y) - g_{n}(y) \right\}$. By Assumption \ref{assumption:abstract_spaces_q} (b), $y_n$ has a converging subsequence with a limit $y_0 \in K^{\hat{q}}$. We obtain that 
		\begin{align*}
		\sup_{y \in K} \left\{ Tf(y) - Tg(y) \right\}& \leq\delta \hat{C}_0(r) + \hat{C}_1(\delta,r) \left( f(y_0) - g(y_0) \right) \\
		&\leq\delta \hat{C}_0(r) + \hat{C}_1(\delta,r) \sup_{y \in K^{\widehat{q}}} \left\{ f(y) - g(y) \right\}.
		\end{align*}
		This establishes strict continuity on bounded sets for $T$.
		
		\smallskip
		
		We proceed with the proof of (b). First note that by strict continuity on bounded sets, and Lemma \ref{lemma:extension_operator}, we can extend $T$ to the quasi-closure of $\cD(T)$, on which $T$ is also strictly continuous on bounded sets.
		
		Next, we show that $\cD(T)$ was in fact quasi-closed to begin with. Let $h$ be in the the strict closure of $\cD(T) \cap B_R$ for some $R > 0$. We prove that $h \in \cD(T) \cap B_R$. Using the extension of $T$ to the quasi-closure of $\cD(T)$, we can define $f := T h$. Let $h_n \in B_n$ such that $\LIM h_n = h$ and $\sup_n \vn{h_n} \leq M \vn{h}$. To establish (b), we need to prove that $\LIM T_n h_n = f$.
		
		\smallskip
		
		On bounded sets, the strict topology coincides with the compact-open topology. Thus, there are functions $h^{K,\varepsilon} \in \cD(T)$ such that $\sup_{K,\varepsilon} \vn{h^{K,\varepsilon}} \leq R$ and 
		\begin{equation*}
		\sup_{y \in K} \left| h(y) - h^{K,\varepsilon}(y) \right|\leq \varepsilon.
		\end{equation*}
		Define $f^{K,\varepsilon} := T h^{K,\varepsilon}$. Furthermore, find $f_n, f_n^{K,\varepsilon}, h_n^{K,\varepsilon}$ such that
		\begin{equation*}
		\LIM f_n = f, \qquad \LIM f_n^{K,\varepsilon} = f^{K,\varepsilon}, \qquad \LIM h_n^{K,\varepsilon} = h^{K,\varepsilon},
		\end{equation*}
		and such that we have an upper bound $r$ for the norms of all involved functions. To establish that $\LIM T_n h_n = f$, it suffices by Lemma \ref{lemma:joint_equicontinuity_with_convergence} to prove that $\LIM T_n h_n - f_n = 0$.
		
		Fix an arbitrary $q$ and $\varepsilon$. Then it suffices to prove that 
		\begin{equation} \label{eqn:bound_on_difference_closedness_set}
		\lim_{n \rightarrow \infty} \sup_{y \in K_n^q} \left|T_n h_n(y) - f_n(y) \right| \leq 4\varepsilon.
		\end{equation}
		
		To do so, fix $\delta > 0$ such that $C_0(r)\delta \leq \varepsilon$.
		
		By strict equi-continuity on bounded sets of the operators $T_n$, starting with $q$, $r$ and $\delta$, we find a $\hat{q}$ and set $K := K^{\hat{q}}$. Note that we can use the same compact set $K$ for the strict continuity estimate for the limiting operator $T$ also.
		
		Using this specific compact set $K$, it follows by the triangle inequality that
		\begin{align*}
		& \left|T_n h_n(y) - f_n(y) \right| \\
		& \quad \leq \left|T_n h_n(y) - T_n h_n^{\varepsilon,K}(y) \right| + \left|T_n h_n^{\varepsilon,K}(y) - f_n^{K,\varepsilon}(y) \right| + \left|f_n^{K,\varepsilon}(y) - f_n(y) \right| \\
		& \quad \leq \left|T_n h_n(y) - T_n h_n^{\varepsilon,K}(y) \right| + \left|T_n h_n^{\varepsilon,K}(y) - f_n^{K,\varepsilon}(y) \right| + \left|f_n^{K,\varepsilon}(y) - f_n(y) \right|
		\end{align*}

		After taking $\limsup_n$ and $\sup_{y \in K_n^q}$ over the three terms separately, we estimate as follows:
		\begin{itemize}
			\item By equi-continuity of the $T_n$ and our choice of $\delta >0$ and $\hat{q}$, we find that
			\begin{equation*}
			\sup_{y \in K_n^q} \left|T_n h_n(y) - T_n h_n^{\varepsilon,K}(y) \right| \leq \varepsilon + \sup_{y \in K_n^{\hat{q}}} \left| h_n(y) -  h_n^{\varepsilon,K}(y) \right|
			\end{equation*}
			As $\LIM h_n = h$ and $\LIM h_n^{\varepsilon,K} = h^{\varepsilon,K}$, and the fact that $\sup_{y \in K} |h(y) - h^{K,\varepsilon}(y)| \leq \varepsilon$ implies that the $\limsup$ over the right-hand side is bounded above by $2 \varepsilon$.
			\item As $h^{K,\varepsilon}$ in $\cD(T)$, we have $\LIM T_n h_n^{K,\varepsilon} = T h^{\varepsilon,K} = f^{\varepsilon,K}$. We also have $\LIM f_n^{K,\varepsilon} = f^{K,\varepsilon}$ so by Lemma \ref{lemma:joint_equicontinuity_with_convergence} the middle term vanishes by choosing $n$ large.
			\item To estimate $ \limsup_n \sup_{y \in K_n^q} \left|f_n^{K,\varepsilon}(y) - f_n(y) \right|$ note that as $\LIM f_n^{K,\varepsilon} = f^{K,\varepsilon}$ and $\LIM f_n = f$, we find
			\begin{equation*}
			\limsup_n \sup_{y \in K_n^q} \left|f_n^{K,\varepsilon}(y) - f_n(y) \right| \leq \sup_{y \in K^q} \left| f^{K,\varepsilon}(y) - f(y) \right| = \sup_{y \in K^q} \left| T h^{K,\varepsilon}(y) - Th(y) \right|
			\end{equation*}
			Using the strict continuity on bounded sets of $T$, we find
			\begin{equation*}
			\limsup_n \sup_{y \in K_n^q} \left|f_n^{K,\varepsilon}(y) - f_n(y) \right| \leq \varepsilon + \sup_{y \in K} \left| h^{K,\varepsilon}(y) - h(y) \right|.
			\end{equation*}
			Thus the right-hand side is bounded by $2 \varepsilon$.
		\end{itemize}
	This establishes \eqref{eqn:bound_on_difference_closedness_set} which concludes the proof of (b). (c) follows by a direct application of Lemma \ref{definition:joint_strict_equicontinuity_boundedsets}.
	\end{proof}

\subsection{An extension of notions of convergence to a space containing additional information}\label{section:general_convergence2}

In the context of problems that involve homogenisation or slow-fast systems, it often pays of to work with multi-valued Hamiltonians whose range naturally takes values in a space of functions with a domain that is larger. This larger domain takes into account a variable that we homogenize over or the `fast' variable. We extend the setting of Section \ref{section:general_convergence1}. We will not need the extension of all results therein, but restrict to the bare essentials.

\begin{assumption} \label{assumption:abstract_spaces2}
	Consider spaces $X_n$ and $X, Y$, and two spaces $\cX,\cY$. We consider Baire measurable maps $\eta_n : X_n \rightarrow \cX$, $\widehat{\eta}_n : X_n \rightarrow \cY$ and Baire measurable injective maps $\eta : X \rightarrow \cX$, $\widehat{\eta} : Y \rightarrow \cY$.  Finally, there are continuous surjective maps $\gamma : Y \rightarrow X$ and $\widehat{\gamma} : \cY \rightarrow \cX$. The maps are such that the following diagram commutes: \\
	\begin{tikzpicture}
	\matrix (m) [matrix of math nodes,row sep=1em,column sep=4em,minimum width=2em]
	{
		{ } & \cY & Y \\
		X_n & { } & { } \\
		{ }  & \cX & X \\};
	\path[-stealth]
	(m-2-1) edge node [above] {$\widehat{\eta}_n$} (m-1-2)
	(m-2-1) edge node [below] {$\eta_n$} (m-3-2)
	(m-1-2) edge node [right] {$\widehat{\gamma}$} (m-3-2)
	(m-1-3) edge node [right] {$\gamma$} (m-3-3)
	(m-1-3) edge node [above] {$\widehat{\eta}$} (m-1-2)
	(m-3-3) edge node [above] {$\eta$} (m-3-2);
	\end{tikzpicture}
	
\end{assumption}


\begin{assumption} \label{assumption:abstract_spaces_q2}
	There is a directed set $\cQ$ (partially ordered set such that every two elements have an upper bound). For each $q \in \cQ$, we have compact sets $K_n^q \subseteq X_n$ a compact sets $K^q \subseteq X$ and $\widehat{K}^q \subseteq Y$ such that
	\begin{enumerate}[(a)]
		\item If $q_1 \leq q_2$, we have $K^{q_1} \subseteq K^{q_2}$, $\widehat{K}^{q_1} \subseteq \widehat{K}^{q_2}$ and for all $n$ we have $K_{n}^{q_1} \subseteq K_n^{q_2}$.
		\item \label{item:assumption_abstract_2_limit_compact} For all $q \in \cQ$ and each sequence $x_n \in K_n^q$, every subsequence of $x_n$ has a further subsequence $x_{n(k)}$ such that $\widehat{\eta}_{n(k)}(x_{n(k)}) \rightarrow \widehat{\eta}(y)$ in $\cY$ for some $y \in \widehat{K}^q$. 
		\item \label{item:assumption_abstract_2_exists_q} For each compact set $K \subseteq X$, there is a $q \in \cQ$ such that
		\begin{equation*}
		\eta(K) \subseteq \liminf_n \eta_n(K_n^q).
		\end{equation*}
		\item \label{item:assumption_abstract_2_gamma_mapsintoeachother} We have $\gamma(\widehat{K}^q) \subseteq K^q$.
	\end{enumerate}
\end{assumption}

Note the subtle difference with Assumption \ref{assumption:abstract_spaces_q} in the sense that here \eqref{item:assumption_abstract_2_limit_compact} is written down in terms of convergence in $\cY$, whereas \eqref{item:assumption_abstract_2_exists_q} is still written down in terms of convergence in $\cX$.

\begin{remark}
	Note that \eqref{item:assumption_abstract_2_limit_compact} and \eqref{item:assumption_abstract_2_gamma_mapsintoeachother} imply that $\eta_{n(k)}(x_{n(k)}) \rightarrow \eta(\gamma(y))$ in $\cX$ with $\gamma(y) \in K^q$. Thus, the Assumption \ref{assumption:abstract_spaces_q2} implies the conditions for $X,X_n,\cX$ for Assumption \ref{assumption:abstract_spaces_q}.

	Thus, in the context of Assumptions \ref{assumption:abstract_spaces2} and \ref{assumption:abstract_spaces_q2}, we can use all notions of the previous sections, if we talk about functions or operators on $X,\cX$ and $X_n$.
\end{remark}

\begin{example}[Reduction of the dimension] \label{example:LIM_dimension_reduction}
	Consider two spaces $X$ and $Z$ and let $Y := X \times Z$, $X_n := X \times Z$ with maps $\eta_n(x,z) = x$, $\hat{\eta}_n(x,z) = (x,z)$ and $\gamma(x,z) = x$.
	
	Assumption \ref{assumption:abstract_spaces_q2} is satisfied for example with $\cQ$ the collection of pairs of compact sets in $X$ and $Z$:
	\begin{equation*}
	\left\{(K_1, K_2) \, \middle| \, \forall \, K_1 \subseteq X, K_2 \subseteq Z \text{ compact}\right\},
	\end{equation*}
	and $K^{(K_1,K_2)}_n = K_1 \times K_2$, $K^{(K_1,K_2)} = K_1$ and $\widehat{K}^{(K_1,K_2)} = K_1 \times K_2$.
	
	We have $\LIM f_n = f$ if and only if $\sup_n \vn{f_n} < \infty$ and for all compact $K_1 \subseteq X$ and $K_2 \subseteq Z$ and sequences $(x_n,z_n) \in K_1 \times K_2$ and $x \in K_X$ such that $x_n \rightarrow x$, we have $f_n(x_n,z_n) \rightarrow f(x)$.
	
	Note that the dependence of $f_n$ on $z_n$ should vanish in the limit.
\end{example}

%
%
%
%
%
%
%

\section{Pseudo-resolvents and Hamilton-Jacobi equations} \label{section:pseudo_resolvent_and_HJ}

Having developed the topological and analytic machinery, we turn to the study of pseudo-resolvents $\{R(\lambda)\}_{\lambda > 0}$ and viscosity solutions for Hamilton-Jacobi equations $f - \lambda H f = h$. In the linear context, the relation between pseudo-resolvents and `classical' solutions to the Hamilton-Jacobi equation is well established. In particular, pseudo-resolvents have been used as tools in the approximation theory of linear semigroups and generators, see for example \cite[Section 3.4]{EN00}. We will show that the linearity of these the involved operators, however, is not essential.

In this section, we will not consider approximation problems yet, and restrict ourselves to the setting $Y = X$ and $\gamma(x) = x$.

\begin{definition}[Pseudo-resolvents] \label{defintion:pseudo_resolvent}
	Consider a space $X$ and a subset $B$ such that $C_b(X) \subseteq B \subseteq M(X)$ on which we have a family of operators $R(\lambda) : B \rightarrow B$, for $\lambda > 0$.  We say that this family is a \textit{pseudo-resolvent} if $R(\lambda) 0 = 0$ for $\lambda > 0$ and if for all $\alpha < \beta$ we have
	\begin{equation*}
	R(\beta) = R(\alpha)\left(R(\beta) - \alpha \frac{R(\beta) - \bONE}{\beta} \right).
	\end{equation*}
\end{definition}

We extend our notion of strict continuity on bounded sets to a collection of operators similar to what we did in Definition \ref{definition:joint_strict_equicontinuity_boundedsets}. Whereas in that Definition we worked with bounded functions only, we extend to unbounded functions in the next definition. In addition, we use that the constants that appear for pseudo-resolvents are typically $1$.

\begin{definition}[Local strict equicontinuity on bounded sets] \label{definition:local_strict_equicont_of_pseudoresolvent}
	Let $B$ be a collection of functions $C_b(X) \subseteq B \subseteq M(X)$. We say that the pseudo-resolvent $R(\lambda) : B \rightarrow B$ is \textit{locally strictly equicontinuous on bounded sets} if for each $\lambda_0 > 0$, each compact set $K \subseteq X$ and $r,\delta > 0$, there is a compact set $\widehat{K} = \widehat{K}(K,\delta,r,\lambda_0)$ such that
	\begin{equation*}
	\sup_{x \in K} |R(\lambda)f(x) - R(\lambda)g(x)| \leq \delta \sup_{x \in X} |f(x) - g(x)| + \sup_{x \in \hat{K}(K,\delta,r,\lambda_0)} |f(x) - g(x)|
	\end{equation*}
	for all $0 <\lambda \leq \lambda_0$, $f,g \in B$ such that $\vn{f} \vee \vn{g} \leq r$.
\end{definition}

Note that the present definition, restricted to a collection $B_0 \subseteq B \cap M_b(X)$, reduces to a definition like the one in Definition \ref{definition:joint_strict_equicontinuity_boundedsets}.

The three main results of this section are
\begin{itemize}
	\item Proposition \ref{proposition:identify_resolvent}, which shows that a pseudo-resolvent $R$ such that $R(\lambda)(f - \lambda H f) = f$ yields viscosity solutions to the Hamilton-Jacobi equation for $H$. That is, shows if the pseudo-resolvent is a left-inverse (classically) of $\bONE - \lambda H$, then it is also a right-inverse (in the viscosity sense). This implies that a pseudo-resolvent can be used to identify the resolvent of $H$.
	\item Proposition \ref{proposition:viscosity_solutions_give_pseudoresolvent} establishes the converse,  viscosity solutions to the Hamilton-Jacobi equation, if unique, can be used to construct a pseudo-resolvent.
	\item Proposition \ref{proposition:Hhat_properties} shows that the pseudo-resolvent of the previous results can be used to define a new operator, which satisfies the conditions of the Crandall-Liggett theorem.
\end{itemize}

Even though, we will always formally think of the Hamilton-Jacobi equation $f - \lambda H f = h$, we will work with two operators $H_\dagger, H_\ddagger$ instead. These operators should be interpreted as a upper and lower bound of the `true' $H$.

Thus, in all sections below we will work with $H_\dagger \subseteq LSC_l(X) \times USC_u(X)$ and $H_\ddagger \subseteq USC_u(X) \times LSC_l(X)$ and study the Hamilton Jacobi equations
	\begin{equation*}
f - \lambda H_\dagger f = h, \qquad f - \lambda H_\ddagger f = h,
\end{equation*}
with $\lambda > 0$ and $h \in M(X)$.

Before proceeding with the announced results, we note that at various points, it is of interest to know whether the domain of definition of the resolvents and operators $B$ includes $C_b(X)$. At least for resolvents, we can under some assumptions include $C_b(X)$ in the domain by continuous extension. We will therefore henceforth work with pseudo-resolvents that have $C_b(X)$ included in their domain of definition.

\begin{lemma} \label{lemma:extension_of_pseudo_resolvent}
	Suppose that $R(\lambda)$ with domains $D_\lambda$ is a pseudo-resolvent that is also locally strictly equi-continuous on bounded sets. Suppose that
	\begin{enumerate}[(a)]
		\item $R(\lambda)$ restricted to $D_\lambda \cap C_b(X)$ maps into $C_b(X)$,
		\item $D_\lambda \cap C_b(X)$ is quasi-dense in $C_b(X)$.
	\end{enumerate}
	Then the restriction of $R(\lambda)$ to $C_b(X)$ can be extended to a pseudo-resolvent that is locally strictly equi-continuous on bounded sets such that for each $\lambda$ the domain of this extension includes $C_b(X)$.
\end{lemma}

\begin{proof}
	By Lemma \ref{lemma:extension_operator}, the restriction of $R(\lambda)$ to $D_\lambda \cap C_b(X)$ can be extended to an operator $\hat{R}(\lambda) : C_b(X) \rightarrow C_b(X)$  that satisfies property \eqref{item:continuity_strict_intermediate} of Proposition \ref{proposition:buc_cont_equiv_to_double_bound}  (with the same choice of constants and compact set as in Definition \ref{definition:local_strict_equicont_of_pseudoresolvent}). The pseudo-resolvent property follows by continuity.
\end{proof}

\subsection{Pseudo-resolvents give solutions to the Hamilton-Jacobi equation}

\begin{proposition} \label{proposition:identify_resolvent}
	Let $B$ be a collection of functions such that $C_b(X) \subseteq B \subseteq M(X)$. Suppose that $R(\lambda) : B \rightarrow B$ is a contractive pseudo-resolvent. In addition, suppose that we have two operators $H_\dagger \subseteq  LSC_l(X) \cap B \times USC_u(X) \cap B$ and $H_\ddagger \subseteq USC_u(X) \cap B \times LSC_l(X) \cap B$.
	
	 Fix $\lambda > 0$ and $h \in M(X)$.
	\begin{enumerate}[(a)]
		\item Let $h \in B$ and let $H_\dagger$ be such that for all $(f,g) \in H_\dagger$ and $0< \varepsilon < \lambda$ we have $f \geq R(\varepsilon)(f - \varepsilon g)$. Then $(R(\lambda) h)^*$ is a viscosity sub-solution to $f - \lambda H_\dagger f = h^*$.
		\item Let $h \in B$ and let $H_\ddagger$ be such that for all $(f,g) \in H_\ddagger$ and $0 < \varepsilon < \lambda$, we have $f \leq R(\varepsilon)(f - \varepsilon g)$. Then $(R(\lambda) h)_*$ is a viscosity super-solution to $f - \lambda H_\ddagger f = h_*$.
	\end{enumerate}
\end{proposition}

To establish this result, we use an auxiliary lemma.

\begin{lemma}[Lemma 7.8 in \cite{FK06}] \label{lemma:technical_to_establish_visc_sol}
	Let $X$ be a some space and let $f,g : X \rightarrow [-\infty,\infty]$ be two functions.
	\begin{enumerate}[(a)]
		\item Suppose there is some $\varepsilon_0 > 0$ such that $f - \varepsilon g \in M(X)$ and $\sup_x f(x) \leq \sup_x f(x) - \varepsilon g(x) < \infty$ for all $0 < \varepsilon < \varepsilon_0$, then there is a sequence $x_n$ in $X$ such that $\lim_n f(x_n) = \sup_x f(x)$ and $\limsup_n g(x_n) \leq 0$.
		\item Suppose there is some $\varepsilon_0 > 0$ such that $f - \varepsilon g \in M(X)$ and $\inf_x f(x) \geq \inf_x f(x) - \varepsilon g(x) > -\infty$ for all $0 < \varepsilon < \varepsilon_0$, then there is a sequence $x_n$ in $X$ such that $\lim_n f(x_n) = \inf_x f(x)$ and $\liminf_n g(x_n) \geq 0$.
	\end{enumerate}
\end{lemma}

\begin{remark}
	The proof in Lemma 7.8 of \cite{FK06} is not correct. In establishing $\lim_n f(x_n) = \sup_x f(x)$ for (a), it is used that $\inf_x g(x) > -\infty$. This claim, however, is not true in general. Consider the following example. Let $X = [0,1]$ and set
	\begin{equation*}
	f(x) = 
	\begin{cases}
	\log x & \text{if } x \neq 0, \\
	-\infty & \text{if } x = 0,
	\end{cases}
	\qquad 
	g(x) =
	\begin{cases}
	\log x & \text{if } x \neq 0, \\
	0 & \text{if } x = 0,
	\end{cases}
	\end{equation*}
	For all $0 < \varepsilon < 1$ we have $\sup_x f(x) = \sup_x f(x) - \varepsilon g(x) = 0$. However, it does not hold that $\inf_x g(x) > - \infty$.
\end{remark}

\begin{proof}
	We only prove (a). Let $0 <\varepsilon_n < \varepsilon_0$ and $\varepsilon_n \rightarrow 0$. For each $n$ pick $x_n$ such that
	\begin{equation} \label{eqn:proof_existence_sequences_basic_inequality}
	f(x_n) \leq \sup_x f(x) \leq \sup_x f(x) - \varepsilon_n g(x) \leq f(x_n) - \varepsilon_n g(x_n) + \varepsilon_n^2.
	\end{equation}
	Combining the outer two-terms leads to $\limsup_n g(x_n) \leq \lim_n \varepsilon_n = 0$.
	
	\smallskip
	
	We now establish that $\lim_n f(x_n) = \sup_x f(x)$. To do so, we first prove $\liminf_n g(x_n) > - \infty$. Fix $\varepsilon \in (0,\varepsilon_0)$. For each $n$, using \eqref{eqn:proof_existence_sequences_basic_inequality}, we have the following chain of inequalities:
	\begin{align*}
	\sup_x f(x) - \varepsilon g(x) & \geq f(x_n) - \varepsilon g(x_n) \\
	& = f(x_n) - \varepsilon_n g(x_n) + \varepsilon_n^2 - (\varepsilon - \varepsilon_n)g(x_n) - \varepsilon^2_n \\
	& \geq \left(\sup_x f(x)\right) - (\varepsilon - \varepsilon_n) g(x_n) - \varepsilon_n^2.
	\end{align*}
	Suppose there is a subsequence $x_{n(k)}$ such that $\lim_k g(x_{n(k)}) = - \infty$, then clearly the right-hand side would diverge to $\infty$. This contradicts the boundedness of the left-hand side. 
	
	From the second and final term in \eqref{eqn:proof_existence_sequences_basic_inequality}, we obtain
	\begin{equation*}
	\left(\sup_x f(x)\right) + \varepsilon_n g(x_n) - \varepsilon_n^2 \leq f(x_n).
	\end{equation*}
	Taking a $\liminf$ on both sides, using that $\liminf_n g(x_n) > -\infty$, we find $\liminf_n f(x_n) \geq \sup_x f(x)$, which establishes $\lim_n f(x_n) = \sup_x f(x)$.
\end{proof}

\begin{proof}[Proof of Proposition \ref{proposition:identify_resolvent}]
	We prove (a). Fix $\lambda > 0$ and $h \in USC_u(X) \cap B$. We prove that $(R(\lambda)h)^*$ is a viscosity solution to $f - \lambda H_\dagger f = h^*$. Fix $(f,g) \in H_\dagger$.

	We use the pseudo-resolvent property of $R$ with $0<\varepsilon < \lambda$, to re-express $R(\lambda)h(x)$. For $f(x)$ we use the assumption in (a). We obtain:
	\begin{align*}
	& \sup_x R(\lambda) h(x) - f(x) \\
	& \leq \sup_x R(\varepsilon) \left(R(\lambda) h(x) - \varepsilon \frac{R(\lambda) h(x) - h(x)}{\lambda}\right) - R(\varepsilon)\left(f(x) - \varepsilon g(x)  \right) \\
	& \leq  \sup_{x} R(\lambda) h(x) - \varepsilon \frac{R(\lambda) h(x) - h(x)}{\lambda} - (f(x) - \varepsilon g(x)), 
	\end{align*}
	where we have used the contractivity of $R(\varepsilon)$ to obtain the final inequality. Note that we use that the domain and range of $H_\dagger$ are contained in $B$ to be able to write down $R(\lambda)$ and $R(\varepsilon)$ applied to $h$ and $f - \varepsilon g$. Next, we take the upper semi-continuous regularization. As $\varepsilon < \lambda$ and $f, f  - \varepsilon g \in LSC_l(X)$, we obtain 
	\begin{equation*}
	\sup_x \left(R(\lambda) h\right)^*(x) - f(x)  \leq  \sup_{x} \left(R(\lambda) h\right)^*(x) - \varepsilon \frac{\left(R(\lambda) h\right)^*(x) - h^*(x)}{\lambda} - (f(x) - \varepsilon g(x)) 
	\end{equation*}
	This establishes (a) of Lemma \ref{lemma:technical_to_establish_visc_sol} which then yields that $(R(\lambda)h)^*$ is a viscosity subsolution to $f - \lambda H_\dagger f = h^*$.
\end{proof}

%
%
%
%

\subsection{Solutions to the Hamilton-Jacobi equation give a pseudo-resolvent}

In this section we prove the converse statement, namely that viscosity solutions give rise to a pseudo-resolvent. We start with an auxiliary lemma that will be key to recovering the pseudo-resolvent property for solutions of viscosity solutions.

\begin{lemma} \label{lemma:viscosity_solutions_give_pseudoresolvent}
	Let $H_\dagger \subseteq LSC_l(X) \times USC_u(X)$ and $H_\ddagger \subseteq USC_u(X) \times LSC_l(X)$. Fix $\lambda > \varepsilon > 0$ and $h \in M(X)$. Then 
	\begin{enumerate}[(a)]
		\item A subsolution $u$ to $f - \lambda H_\dagger f = h$ is a subsolution to $f - \varepsilon H_\dagger f = u - \varepsilon \tfrac{u - h}{\lambda}$.
		\item A supersolution $v$ to $f - \lambda H_\dagger f = h$ is a supersolution to $f - \varepsilon H_\dagger f = v - \varepsilon \tfrac{v - h}{\lambda}$.
	\end{enumerate}
\end{lemma}

\begin{proof}
	We prove (a) only. Let $(f,g) \in H_\dagger$ and let $x_n$ be sequence such that $\lim_n u(x_n) - f(x_n) = \sup_x u(x) - f(x)$ and $\limsup_n u(x_n) - \lambda g(x_n) -h(x_n) \leq 0$. Because
	\begin{equation*}
	u(x_n) - \varepsilon g(x_n) - \left(u(x_n) - \varepsilon \frac{u(x_n) - h(x_n)}{\lambda}\right) = \frac{\varepsilon}{\lambda} \left( u(x_n) - \lambda g(x_n) -h(x_n) \right),
	\end{equation*}
	it follows that $u$ is a viscosity subsolution to $f - \varepsilon H_\dagger f = u - \varepsilon \tfrac{u - h}{\lambda}$.
\end{proof}

	\begin{proposition} \label{proposition:viscosity_solutions_give_pseudoresolvent}
		Let $H_\dagger \subseteq LSC_l(X) \times USC_u(X)$ and $H_\ddagger \subseteq USC_u(X) \times LSC_l(X)$. Suppose that for each $\lambda  > 0$ and $h \in C_b(X)$ the comparison principle holds for 
		\begin{equation} \label{eqn:pair_of_HJ}
		f - \lambda H_\dagger f = h, \qquad f - \lambda H_\ddagger f = h.
		\end{equation}
		Suppose that there exists a viscosity solution to \eqref{eqn:pair_of_HJ} for all $\lambda > 0$ and $h \in C_b(X)$. Denote this unique solution by $R(\lambda)h$. Then $R(\lambda)$ forms a contractive pseudo-resolvent on $C_b(X)$. 
	\end{proposition}

\begin{remark}
The proposition has the drawback that we have to assume the comparison principle for all $h \in C_b(X)$ and $\lambda > 0$. In practice, often one only has the comparison principle for a dense set of functions. The result does show that we should expect the pseudo-resolvent property to hold. Thus, if one has a method to produce viscosity solutions $R(\lambda)h$ via some other means, like a limiting procedure, or an explicit formula, one can aim to prove the pseudo-resolvent property directly.
\end{remark}

\begin{proof}[Proof of Proposition \ref{proposition:viscosity_solutions_give_pseudoresolvent}]
	Let $0 < \varepsilon < \lambda$ and $h \in C_b(X)$. Let $R(\lambda)h$ be the unique viscosity solution to $f - \lambda H_\dagger f = h$ and $f - \lambda H_\ddagger f = h$. Note that by the comparison principle $R(\lambda)h \in C_b(X)$. 
	
	By the comparison principle $R(\lambda)$ is contractive. We next establish that $R(\lambda)$ is a pseudo-resolvent. By Lemma \ref{lemma:viscosity_solutions_give_pseudoresolvent}, we find for $h \in D$ that $R(\lambda) h$ is a viscosity solution to 
	\begin{equation*}
	f - \varepsilon H_\dagger f = R(\lambda)h - \varepsilon\frac{R(\lambda)h - h}{\lambda_0}, \qquad f - \varepsilon H_\ddagger f = R(\lambda)h - \varepsilon\frac{R(\lambda)h - h}{\lambda}.
	\end{equation*}
	As $h$ and $R(\lambda) h \in C_b(X)$, we find by the comparison principle for \eqref{eqn:pair_of_HJ} with $\varepsilon$ instead of $\lambda$ that
	\begin{equation*}
	R(\lambda) h = R(\varepsilon) \left(R(\lambda) h - \varepsilon \frac{R(\lambda) h - h}{\lambda}\right).
	\end{equation*}
\end{proof}

\subsection{Defining a Hamiltonian using a pseudo-resolvent}

We proceed by showing that a pseudo-resolvent can be used to define a Hamiltonian, so that the pseudo-resolvent gives both classical and viscosity solutions to the associated Hamilton-Jacobi equation. Thus, this newly defined operator satisfies the conditions for the Crandall-Liggett theorem, and can be used for approximation arguments.

\begin{proposition} \label{proposition:Hhat_properties}
	Let $R(\lambda) : C_b(X) \rightarrow C_b(X)$ be a contractive pseudo-resolvent. Define the operator
	\begin{equation*}
	\widehat{H} = \left\{\left(R(\lambda)h , \frac{R(\lambda)h - h}{\lambda}\right) \, \middle| \, \lambda > 0, h \in C_b(E) \right\}
	\end{equation*}
	For all $h \in C_b(E)$ and $\lambda > 0$:
	\begin{enumerate}[(a)]
		\item For all $\lambda > 0$ and $h \in C_b(X)$ the comparison principle holds for $f - \lambda \widehat{H}f = h$, and $R(\lambda)h$ is the unique viscosity and classical solution.
		\item $\widehat{H}$ is dissipative and satisfies the range condition.
	\end{enumerate}
\end{proposition}

The proposition is mainly useful in combination with Proposition \ref{proposition:viscosity_solutions_give_pseudoresolvent}. Namely, viscosity solutions for \eqref{eqn:pair_of_HJ} can be used to define a contractive pseudo-resolvent. Which by this proposition can be used to define a new operator that satisfies the conditions for the Crandall-Liggett result. Alternatively, one constructs a contractive pseudo-resolvent via an approximation argument as we will do below in Section \ref{section:convergence_of_operators_and_inverses}.

\smallskip

We start by establishing a natural property for sub- and supersolutions of Hamilton-Jacobi equations. We mention it separately for later use.

\begin{lemma} \label{lemma:identification:viscosity_subsol}
	Let $H_\dagger \subseteq LSC_l(X) \times USC_u(X)$ and $H_\ddagger \subseteq USC_u(X) \times LSC_l(X)$. Fix $\varepsilon > 0$.
	\begin{enumerate}[(a)]
		\item Let $(f_0,g_0) \in H_\dagger$. Suppose that $\hat{f}$ is a viscosity subsolution to $f - \varepsilon H_\dagger f = f_0 - \varepsilon g_0$, then $\hat{f} \leq f_0$. 
		\item Let $(f_0,g_0) \in H_\ddagger$. Suppose that $\hat{f}$ is a viscosity supersolution to $f - \varepsilon H_\ddagger f = f_0 - \varepsilon g_0$, then $\hat{f} \geq f_0$. 
	\end{enumerate}
\end{lemma}

\begin{proof}
	We only prove (a). Fix $\varepsilon > 0$, $(f_0,g_0) \in H_\dagger$ and let $\hat{f}$ be a viscosity subsolution to $f - \varepsilon H_\dagger f = f_{0} - \epsilon g_0$. Then there is a sequence $x_n$ such that
	\begin{equation*} 
	\lim_n \hat{f}(x_n) - f_0(x_n) = \sup_x \hat{f}(x) - f_0(x)
	\end{equation*}
	and
	\begin{equation*}
	\limsup_n \hat{f}(x_n) - \varepsilon g_0(x_n) - (f_0 - \varepsilon g_0)(x_n) \leq 0.
	\end{equation*}
	We obtain $\sup_x \hat{f}(x) - f_0(x) = \limsup_n \hat{f}(x_n) - f_0(x_n)  \leq 0$ establishing the claim.
\end{proof}

\begin{proof}[Proof of Proposition \ref{proposition:Hhat_properties}]

	We prove the comparison principle. Fix $\lambda >0$ and $h_1,h_2 \in D$. By construction $R(\lambda)h_1$ and $R(\lambda)h_2$ solve $f - \lambda \widehat{H}f = h_i$ classically.	 Let $u$ be a subsolution to $f - \lambda \widehat{H} f = h_1$ and $v$ a supersolution to $f - \lambda \widehat{H} f = h_2$. By Lemma \ref{lemma:identification:viscosity_subsol} (a) for $\widehat{H}$ instead of $H_\dagger$ and $(f_0,g_0) = \left(R(\lambda)h_1, \tfrac{R(\lambda)h_1-h_1}{\lambda}\right)$, we find $u \leq R(\lambda) h_1$. Because $R(\lambda)$ is contractive, we find
	\begin{equation*}
	\sup_x u(x) - v(x) \leq \sup_x R(\lambda)h_1(x) - R(\lambda)h_2(x) \leq \sup_x h_1(x) - h_2(x),
	\end{equation*}
	establishing the comparison principle for $f - \lambda \widehat{H}f = h$.
	
	Next, we prove that $R(\lambda)h$ is a viscosity subsolution to $f - \lambda \widehat{H} f = h$ with an argument similar to that of Proposition \ref{proposition:identify_resolvent}. Pick $R(\mu) h_0 \in \cD(\widehat{H})$. By the pseudo-resolvent property of $R$, see Proposition \ref{proposition:viscosity_solutions_give_pseudoresolvent}, and the contractivity of $R$, we find for $0 < \varepsilon < \lambda \wedge \mu$ that
	\begin{align*}
	& \sup_x R(\lambda)h(x) - R(\mu)h_0(x) \\
	& = \sup_x R(\varepsilon) \left(R(\lambda)h - \varepsilon \frac{R(\lambda)h - h}{\lambda} \right)(x) - R(\varepsilon) \left(R(\mu)h_0 - \varepsilon \frac{R(\mu)h_0 - h_0}{\mu}\right)(x) \\
	& \leq \sup_x R(\lambda)h(x) - \varepsilon \frac{R(\lambda)h(x) - h(x)}{\lambda} -  \left(R(\mu)h_0(x) - \varepsilon \frac{R(\mu)h_0(x) - h_0(x)}{\mu} \right). 
	\end{align*}
	By Lemma \ref{lemma:technical_to_establish_visc_sol}, we conclude that $R(\lambda)h$ is a viscosity subsolution to $f - \lambda \widehat{H} f = h$. The super-solution property follows similarly. 
	
	\smallskip
	
	Thus, by the comparison principle, $R(\lambda)h$ is the unique viscosity solution to $f - \lambda Hf = h$. Finally, suppose that $f_0$ is another classical solution to $f - \lambda \widehat{H} f = h$. Thus, there is a $g_0$ such that $(f_0,g_0) \in \widehat{H}$ and $f_0 - \lambda g_0 = h$. As $R(\lambda)f$ is a viscosity solution to $f - \lambda \widehat{H}f = h$, and hence to $f - \lambda \widehat{H}f = f_0 - \lambda g_0$, we find again by Lemma \ref{lemma:identification:viscosity_subsol} that $f_0 = R(\lambda) h$. 
	
	\smallskip

	Finally, we establish (b). Note that the range condition for $\widehat{H}$ is satisfied by construction. We establish dissipativity. By construction, there is some $\lambda > 0$ and $h \in C_b(X)$ such that $f_1 = R(\lambda)h$ and $g_1 = \lambda^{-1}(f_1 - h)$. As $R(\lambda) h$ is a viscosity subsolution to $f - \lambda \widehat{H}f = h$ by (a), there are $x_n$ such that
	\begin{gather*}
	\lim_n f_1(x_n) - f_2(x_n) = \sup_x f_1(x) - f_2(x), \\
	\limsup_n g_1(x_n) - g_2(x_n) \leq 0.
	\end{gather*}
	This implies for all $\mu >0$ that
	\begin{align*}
	\sup_x f_1(x) - \mu g_1(x) - (f_2(x) - \mu g_2(x) ) & \geq \limsup_n f_1(x_n) - \mu g_1(x_n) - (f_2(x_n) - \mu g_2(x_n) ) \\
	& = \lim_n f_1(x_n) - f_2(x_n) = \sup_x f_1(x) - f_2(x).
	\end{align*}
	A similar argument using the supersolution property yields the other inequality for the infima. We conclude that for all $\mu > 0$:
	\begin{equation*}
	\vn{f_1 - \mu g_1 - (f_2 - \mu g_2)} \geq \vn{f_1 - f_2}
	\end{equation*}
	establishing dissipativity of $\widehat{H}$.
\end{proof}

\subsection{The pseudo-resolvent yields viscosity solutions via a density argument}

We introduce a final tool in the study of pseudo-resolvents $R(\lambda)$ and viscosity solutions to Hamilton-Jacobi equations 
\begin{equation} \label{eqn:pair_of_HJ2}
f - \lambda H_\dagger f = h_1, \qquad f -\lambda H_\dagger f = h_2.
\end{equation}

In Proposition \ref{proposition:viscosity_solutions_give_pseudoresolvent}, we showed that if we can solve \eqref{eqn:pair_of_HJ2} in the viscosity sense for all $\lambda > 0$ and $h \in C_b(X)$, then the comparison principle is sufficient to establish that the solutions $R(\lambda)$ form a contractive pseudo-resolvent.

Often, however, one can construct a pseudo-resolvent $R(\lambda)$ such that $R(\lambda)h$ solves \eqref{eqn:pair_of_HJ2} in the viscosity sense for $\lambda > 0$ and $h = h_1 = h_2 \in D$, where $D \subseteq C_b(X)$ is quasi-dense.  Thus, the main step to establish the pseudo-resolvent property in the proof cannot be carried out. This happens for example in the construction in Theorem \ref{theorem:convergence_of_resolvents}. Even though there we can establish the pseudo-resolvent property by approximation, this situation is not completely satisfying.

The result below gives an alternative that does not need an explicit form for the resolvent, or that it is the limit of a sequence of pseudo-resolvents.

A second reason to establish that $R(\lambda)h$ gives viscosity solutions for all $h \in C_b(X)$ is that this property can be used as input for follow-up arguments, see e.g. Condition \ref{condition:convergence_of_generators_and_conditions_extended_supsuperlim} below.

The argument below is based on compactness and quasi-density of $D$ in $C_b(X)$.

\begin{proposition} \label{proposition:extending_set_of_viscosity_solutions}
	Let $X$ be a space in which compact sets are metrizable. For each $\lambda > 0$ let $R(\lambda) : C_b(X) \rightarrow C_b(X)$ be an operator that is strictly continuous on bounded sets. Let $D$ be quasi-dense in $C_b(X)$.

	Then (i) and (ii) hold.
	\begin{enumerate}[(i)]
		\item Let $H_\dagger,\widetilde{H}_\dagger \subseteq LSC_l(X) \times USC_u(X)$ be two operators such that
		\begin{enumerate}[(a)]
			\item For each $h \in C_b(X)$ and $\lambda > 0$, if $f$ is a viscosity sub-solution to $f - \lambda H_\dagger f = h$ then it is a viscosity subsolution to $f - \lambda \widetilde{H}_\dagger f = h$.
			\item Each function $f \in \cD(\widetilde{H}_\dagger)$ has compact sub-levelsets. 
		\end{enumerate}
		Suppose that $R(\lambda)h$ is a viscosity subsolution to $f - \lambda H_\dagger f = h$ for all $h \in D$ and $\lambda > 0$. Then $R(\lambda)$ is a viscosity subsolution to $f - \lambda \widetilde{H}_\dagger f = h$ for all $h \in C_b(X)$ and $\lambda > 0$.
		\item Let $H_\ddagger,\widetilde{H}_\ddagger \subseteq USC_u(X) \times LSC_l(X)$ be two operators such that
		\begin{enumerate}[(a)]
			\item For each $h \in C_b(X)$ and $\lambda > 0$, if $f$ is a viscosity super-solution to $f - \lambda H_\ddagger f = h$ then it is a viscosity supersolution to $f - \lambda \widetilde{H}_\ddagger f = h$.
			\item Each function $f \in \cD(\widetilde{H}_\ddagger)$ has compact super-levelsets.
		\end{enumerate}
		Suppose that $R(\lambda)h$ is a viscosity supersolution to $f - \lambda H_\ddagger f = h$ for all $h \in D$ and $\lambda > 0$. Then $R(\lambda)$ is a viscosity supersolution to $f - \lambda \widetilde{H}_\ddagger f = h$ for all $h \in C_b(X)$ and $\lambda > 0$.
	\end{enumerate}
\end{proposition}

\begin{remark}
	The conditions sketched above are satisfied in a wide range of situations. Consider for example a Hamiltonian $\cH : \bR \times \bR \rightarrow \bR$ in terms of location $x \in \bR$ and momentum $p \in \bR$ (the argument easily extends to e.g. manifolds). We assume that $p \mapsto \cH(x,p)$ is convex for all $x$.
	
	Assume there is a continuously differentiable function $\Upsilon$ that has compact sub-level sets and is such that $\sup_x \cH(x, \Upsilon'(x)) \leq c$ for some $c \in \bR$.
	
	The condition in (i) then holds for the operator $(H_\dagger,\cD(H_\dagger))$ defined by
	\begin{equation*}
	\cD(H_\dagger) := C_b^1(\bR), \qquad \forall \, f \in C_b^1(\bR): H_\dagger f(x) = \cH(x,f'(x)),
	\end{equation*}
	and the operator $(\widetilde{H}_\dagger,\cD(\widetilde{H}_\dagger))$ defined by
	\begin{gather*}
	\cD(\widetilde{H}_\dagger) := \bigcup_{\varepsilon \in (0,1)} \left\{ (1-\varepsilon)f + \varepsilon \Upsilon \, \middle| \, f \in C_b^1(\bR) \right\} \\
	\forall \, \widetilde{f} := (1-\varepsilon) f +  \varepsilon \Upsilon: \qquad \widetilde{H}_\dagger \widetilde{f}(x) = (1-\varepsilon) H_\dagger f(x) + \varepsilon c.
	\end{gather*}
	For a proof of (a) one uses convexity and e.g. methods like Lemmas 7.6 and 7.7 in \cite{FK06}. For an application of these lemmas, see e.g. Section A.2 in \cite{CoKr17}.
\end{remark}

\begin{proof}[Proof of Proposition \ref{proposition:extending_set_of_viscosity_solutions}]
	We only prove (a). Fix $h \in C_b(X)$ and $\lambda > 0$. 
	
	\smallskip	
	
	Pick $(f,g) \in \widetilde{H}_\dagger$. As $f$ has compact sub-level sets, the set $K := \left\{x \, \middle| \, f(x) \leq 2\vn{h} \right\}$ is compact. By quasi-density of $D$ in $C_b(X)$, there are $h_n$ such that $\sup_n \vn{h_n} \leq 2 \vn{h}$ and $\sup_{x \in K} |h(x) - h_n(x)| \leq n^{-1}$.
	
	By the viscosity subsolution property, the fact that $f$ has compact level-sets, and $\sup_n \vn{h_n} \leq 2\vn{h}$, there are $x_n \in K$ such that
	\begin{equation} \label{eqn:extend_viscosity_solution_via_compact_levelsets}
	\begin{split}
	R(\lambda)h_n(x_n) - f(x_n) = \sup_x R(\lambda)h_n(x) - f(x), \\
	R(\lambda)h_n(x_n) - \lambda g(x_n) \leq h_n(x_n).
	\end{split}
	\end{equation}
	Because $K$ is compact and metrizable, we can assume without loss of generality that $x_n$ converges to $x_0$. $h_n$ converges uniformly to $h$ on $K$. As $R(\lambda)$ is strictly continuous on bounded sets also $R(\lambda)h_n$ converges uniformly on $K$ to $R(\lambda) h$. Thus, we can take limit in the first equation and $\limsup$ in the second equation of \eqref{eqn:extend_viscosity_solution_via_compact_levelsets} to obtain
	\begin{equation*} 
	\begin{split}
	R(\lambda)h(x_0) - f(x_0) = \sup_x R(\lambda)h(x) - f(x) \\
	R(\lambda)h_n(x_0) - \lambda g(x_0) \leq h(x_0).
	\end{split}
	\end{equation*}
	Note that we used that $g$ is upper semi-continuous to obtain the correct inequality. These two equations establish that $R(\lambda) h$ is a viscosity solution  $f - \lambda \widetilde{H}_\dagger f = h$.
\end{proof}

\section{Convergence of resolvents} \label{section:convergence_of_operators_and_inverses}

We now turn to the main question of the paper: that of approximation. Our first goal is to establish that viscosity solutions to $f - \lambda H_n f = h$ converge to a viscosity solution of the equation $f - \lambda H f = h$. All arguments will be based in the context of converging spaces. 

As mentioned in the introduction some problems, e.g. slow-fast or multi-scale systems lead to natural limiting Hamiltonians that are multi-valued as a graph $H \subseteq C_b(X) \times C_b(Y)$, where $Y$ is some larger space that takes into account the fast variable or the additional scales. 
The notions of convergence of functions that are relevant have been introduced in Sections \ref{section:general_convergence1} and \ref{section:general_convergence2}.

\subsection{From convergence of Hamiltonians to convergence of resolvents}

	A first notion of a limit of Hamiltonians is given by the notion of an extended limit. This notion is essentially the extension of the convergence condition for generators from the setting of the Trotter-Kato approximation theorem to a more general context. The generalization is made to include operators defined on different spaces, and is also applicable to non-linear operators as well. See e.g. the works of Kurtz and co-authors \cite{EK86,Ku70,Ku73,FK06}.
	
	We define this notion for the setting in which $X = Y$.
	
	\begin{definition} \label{definition:extended_limit}
		Consider the setting of Assumptions \ref{assumption:abstract_spaces} and \ref{assumption:abstract_spaces_q}. Suppose that for each $n$ we have an operator $H_{n} \subseteq M_b(X_n) \times M_b(X_n)$. The \textit{extended limit} $ex-\LIM_n H_{n}$ is defined by the collection $(f,g) \in M_b(X)\times M_b(X)$ such that there exist $(f_n,g_n) \in H_{n}$ with the property that $\LIM_n f_n = f$ and $\LIM_n g_n = g$.
	\end{definition}
	
	We aim to have a more flexible notion of convergence by replacing all operators $H_n$ and $H$ by operators $(H_{n,\dagger},H_{n,\ddagger},H_\dagger,H_\ddagger)$ that intuitively form natural upper and lower bounds for $H_n$ and $H$. We will also generalize by considering limiting Hamiltonians that take values in the set of functions on $Y$ instead of $X$.

	\begin{definition} \label{definition:extended_sub_super_limit}
		Consider the setting of Assumptions \ref{assumption:abstract_spaces2} and \ref{assumption:abstract_spaces_q2}.  Suppose that for each $n$ we have two operators $H_{n,\dagger} \subseteq LSC_l(X_n) \times USC_u(X_n)$ and $H_{n,\ddagger} \subseteq USC_u(X_n) \times LSC_l(X_n)$.
		\begin{enumerate}[(a)]
			\item The \textit{extended sub-limit} $ex-\subLIM_n H_{n,\dagger}$ is defined by the collection $(f,g) \in H_\dagger \subseteq LSC_l(X) \times USC_u(Y)$ such that there exist $(f_n,g_n) \in H_{n,\dagger}$ satisfying
			\begin{gather} 
			\LIM f_n \wedge c = f \wedge c, \qquad \forall \, c \in \bR, \label{eqn:convergence_condition_sublim_constants} \\
			\sup_{n} \sup_{x \in X_n} g_n(x) < \infty, \label{eqn:convergence_condition_sublim_uniform_gn}
			\end{gather}
			and if for any $q \in \cQ$ and sequence $z_{n(k)} \in K_{n(k)}^q$ (with $k \mapsto n(k)$ strictly increasing) such that $\lim_{k} \widehat{\eta}_{n(k)}(z_{n(k)}) = \widehat{\eta}(y)$ in $\cY$ with $\lim_k f_{n(k)}(z_{n(k)}) = f(\gamma(y)) < \infty$ we have
			\begin{equation} \label{eqn:sublim_generators_upperbound}
			\limsup_{k \rightarrow \infty}g_{n(k)}(z_{n(k)}) \leq g(y).
			\end{equation}
			\item The \textit{extended super-limit} $ex-\superLIM_n H_{n,\ddagger}$ is defined by the collection $(f,g) \in H_\ddagger \subseteq USC_u(X) \times LSC_l(Y)$ such that there exist $(f_n,g_n) \in H_{n,\ddagger}$ satisfying
			\begin{gather} 
			\LIM f_n \vee c = f \vee c, \qquad \forall \, c \in \bR, \label{eqn:convergence_condition_superlim_constants} \\
			\inf_{n} \inf_{x \in X_n} g_n(x) > - \infty, \label{eqn:convergence_condition_superlim_uniform_gn}
			\end{gather}
			and if for any $q \in \cQ$ and sequence $z_{n(k)} \in K_{n(k)}^q$ (with $k \mapsto n(k)$ strictly increasing) such that $\lim_{k} \widehat{\eta}_{n(k)}(z_{n(k)}) = \widehat{\eta}(y)$ in $\cY$ with $\lim_k f_{n(k)}(z_{n(k)}) = f(\gamma(y)) > - \infty$ we have
			\begin{equation}\label{eqn:superlim_generators_lowerbound}
			\liminf_{k \rightarrow \infty}g_{n(k)}(z_{n(k)}) \geq g(y).
			\end{equation}
		\end{enumerate}
		
	\end{definition}

	\begin{remark}
		The conditions in \eqref{eqn:convergence_condition_sublim_uniform_gn} and \eqref{eqn:sublim_generators_upperbound} are implied by $\LIMSUP_n g_n \leq g$ and \eqref{eqn:convergence_condition_superlim_uniform_gn} and \eqref{eqn:superlim_generators_lowerbound} are implied by $\LIMINF_n g_n \geq g$.
		
		It is not clear to the author whether a weakened symmetric statement in which \eqref{eqn:convergence_condition_sublim_constants} is replaced by $\LIMINF f_n \geq f$ and \eqref{eqn:convergence_condition_superlim_constants} by $\LIMSUP f_n \leq f$ is possible.
	\end{remark}

	\begin{remark}
		The notion of $ex-\subLIM$ and $ex-\superLIM$ follows closely Condition 7.11 \cite{FK06}. Note, however, that our definition does away with the first conditions in (7.19) and (7.22), which in \cite{FK06} are used in a crucial way in controlling the approximation of $H_n$ by operators $H_n^\varepsilon$ that are constructed from the Yosida approximant $A_n^\varepsilon$ of the linear operator $A_n$. 
	\end{remark}

	Given our main conditions on upper and lower bounds for sequences of Hamiltonians, we can state the main condition for our approximation result.

	\begin{condition} \label{condition:convergence_of_generators_and_conditions_extended_supsuperlim}
		Consider the setting of Assumptions \ref{assumption:abstract_spaces2} and \ref{assumption:abstract_spaces_q2}. 
		
		There are sets $B_n \subseteq M(X_n)$, contractive pseudo-resolvents $R_n(\lambda) : B_n \rightarrow B_n$, $\lambda >0$, operators
		\begin{align*}
		H_{n,\dagger} & \subseteq LSC_l(X_n) \cap B_n  \times USC_u(X_n) \cap B_n, \\
		H_{n,\ddagger} & \subseteq USC_u(X_n) \cap B_n \times LSC_l(X_n) \cap B_n, 
		\end{align*}
		and
		\begin{equation*}
		H_\dagger  \subseteq LSC_l(X) \times USC_u(Y),  \qquad H_\ddagger  \subseteq USC_u(X) \times LSC_l(Y).
		\end{equation*}
		These spaces and operators have the following properties:
		\begin{enumerate}[(a)]
			\item \label{item:approximation_of_functions} There is a $M > 0$ such that for each $h \in C_b(X)$ there are $h_n \in B_n$ such that $\LIM h_n = h$ and $\sup_n \vn{h_n} \leq M \vn{h}$
			\item \label{item:convH_sub_superLIM} $H_\dagger \subseteq ex-\subLIM H_{n,\dagger}$ and $H_\ddagger \subseteq ex-\superLIM H_{n,\ddagger}$;
			\item \label{item:convH_pseudoresolvents_solve_HJ} For each $n \geq 1$, $\lambda > 0$ and $h \in B_n$ the function $(R_n(\lambda)h)^*$ is a viscosity subsolution to $f -  \lambda H_{n,\dagger} = h$. Similarly, $(R_n(\lambda)h)_*$ is a viscosity supersolution to $f - \lambda  H_{n,\ddagger} f = h$.
			\item \label{item:convH_strict_equicont_resolvents} We have local strict equi-continuity on bounded sets: for all $q \in \cQ$, $\delta > 0$ and $\lambda_0 > 0$, there is a $\hat{q} \in \cQ$ such that for all $n$ and $h_{1},h_{2} \in B_n$ and $0 < \lambda \leq \lambda_0$ that
			\begin{equation*}
			\sup_{y \in K_n^q} \left\{ R_n(\lambda)h_{1}(y) - R_n(\lambda)h_{2}(y) \right\} \leq \delta \sup_{x \in X_n} \left\{ h_{1}(x) - h_{2}(x) \right\} + \sup_{y \in K^{\hat{q}}_n} \left\{ h_{1}(y) - h_{2}(y) \right\}.
			\end{equation*}
		\end{enumerate}
	\end{condition}
	
	 \begin{remark}
	 	As a follow up on Remark \ref{remark:balance_LIM_and_equi_cont}, note that \eqref{item:convH_sub_superLIM} and \eqref{item:convH_strict_equicont_resolvents} reflect a careful balance. Proving \eqref{item:convH_strict_equicont_resolvents} will be relatively easy if one chooses large sets for $K^q_n$, which leads to a  difficulties in \eqref{item:convH_sub_superLIM}. Context specific knowledge is needed for a proper choice.
	 \end{remark}

	We briefly discuss the relevance of our four conditions and a sketch of the proof.
	
	\begin{itemize}
		\item Conditions \eqref{item:convH_sub_superLIM} and \eqref{item:convH_pseudoresolvents_solve_HJ} are aimed at showing that $\LIMSUP_n R_n(\lambda) h_n$ yields a viscosity subsolution to $f - \lambda H_\dagger f = h$, whereas $\LIMINF_n R_n(\lambda) h_n$ yields a viscosity supersolution to $f - \lambda H_\ddagger f = h$ if  $\LIM h_n = h$. In combination with the comparison principle for $h$ in a quasi-dense set $D$, we obtain a viscosity solution for $h \in D$ that we call $R(\lambda)h$. In addition, we obtain that $\LIM R_n(\lambda)h_n = R(\lambda)h$. 
		\item Using Proposition \ref{proposition:extending_operators_via_LIM}, the operator $R(\lambda)$ extends to $C_b(X)$ on which it is strictly continuous on bounded sets. In particular the operator is contractive.
		\item The operator $R(\lambda)$ is a pseudo-resolvent as it is the limit of pseudo-resolvents.
	\end{itemize}
	
	 Some technical difficulties need to be settled. The main idea for our the first step of our strategy is to apply the method that was also used in the proof of Proposition \ref{proposition:viscosity_solutions_give_pseudoresolvent}. Here we used contractivity of the resolvent and Lemma \ref{lemma:technical_to_establish_visc_sol}. In this setting, we need to take care of our special notion of $\LIM$. Thus, we need to replace contractivity by control along compact subsets $K_n^q$ for a fixed $q$. This is the main aim of Condition \eqref{item:convH_strict_equicont_resolvents}.

\begin{theorem} \label{theorem:convergence_of_resolvents}
	Let Condition \ref{condition:convergence_of_generators_and_conditions_extended_supsuperlim} be satisfied. Let $D \subseteq C_b(X)$ be quasi-dense in $C_b(X)$. Suppose that for each $\lambda > 0$ and $h \in D$ the comparison principle holds for 
	\begin{equation} \label{eqn:pair_of_HJ_approx_theorem}
	f - \lambda H_\dagger f = h, \qquad f - \lambda H_\ddagger f = h.
	\end{equation}
	Then there is a collection of operators $R(\lambda) : C_b(X) \rightarrow C_b(X)$ such that
	\begin{enumerate}[(a)]
		\item For each $h \in D$ and $\lambda > 0$ the function $R(\lambda)h$ is a viscosity subsolution to $f - \lambda H_\dagger f = h$ and a viscosity supersolution to $f - \lambda H_\ddagger f = h$.
		\item The operators are locally strictly equi-continuous on bounded sets.
		\item The operators form a pseudo-resolvent.
		\item For $\lambda > 0$, $h_n \in B_n$ and $h \in C_b(X)$ such that $\LIM h_n = h$, we have $\LIM R_n(\lambda) h_n = R(\lambda) h$.
	\end{enumerate}
\end{theorem}

	We state the main argument for the theorem as a separate proposition as it is valid in a context that goes slightly beyond the theorem.
	
	\begin{proposition} \label{proposition:existence_viscosity_sub_super_solutions}
		Let Condition \ref{condition:convergence_of_generators_and_conditions_extended_supsuperlim} be satisfied. Let $h_n \in M(X_n)$. 
		\begin{enumerate}[(a)]
			\item Let $h_n \in M(X_n)$, $h \in USC_u(X)$ and suppose that $\LIMSUP_n h_n \leq h$. Define
			\begin{equation*}
			\overline{F} := \LIMSUP_n R_n(\lambda) h_n,
			\end{equation*}
			then $\overline{F}^*$ is a viscosity subsolution to $f - \lambda H_\dagger f = h$.
			\item Let $h_n \in M(X_n)$, $h \in LSC_l(X)$ and suppose that $\LIMINF_n h_n \geq h$. Define
			\begin{equation*}
			\underline{F} := \LIMINF_n R_n(\lambda) h_n,
			\end{equation*}
			then $\underline{F}^*$ is a viscosity supersolution to $f - \lambda H_\ddagger f = h$.
		\end{enumerate} 
		
	\end{proposition}
	
	It should be noted that the proof of this proposition does not use Condition \ref{condition:convergence_of_generators_and_conditions_extended_supsuperlim} \eqref{item:approximation_of_functions}.

	The main idea of the proof of the proposition is based on the proof of Lemma 7.14 in \cite{FK06}, but improves on this result in terms of the three properties mentioned in the introduction: applicability outside of the context of large deviations, operators $H_{n,\dagger},H_{n,\ddagger}$ instead of $H_n$ and the possibility to work in $\cX$ instead of in $X$.

		\begin{proof}[Proof of Proposition \ref{proposition:existence_viscosity_sub_super_solutions}]
			We only prove (a). Note first of all that by contractivity of $R_n(\lambda)$, we have
			\begin{equation} \label{eqn:upper_bound_resolvents}
			\sup_{n} \sup_{x \in X_n} R_n(\lambda)h_n(x) \leq \sup_{n} \sup_{x \in X_n} h_n(x) < \infty,
			\end{equation}
			so that we can indeed write down $\overline{F} := \LIMSUP R_n(\lambda) h_n$.
			
			\smallskip
			
			We prove that $\overline{f} := \overline{F}^*$ is a viscosity subsolution of $f - \lambda H_\dagger f = h$. First of all, $\overline{f}$ is upper semi-continuous by construction. Second, $\overline{f}$ is bounded from above as seen above as a consequence of \eqref{eqn:upper_bound_resolvents}. We will prove that $\overline{f}$ also satisfies the final property of the definition of subsolutions. As in the proof of Proposition \ref{proposition:identify_resolvent}, we use Lemma \ref{lemma:technical_to_establish_visc_sol}(a).
			
			Thus, for $(f_0,g_0) \in H_\dagger$ it suffices to prove that for $0 < \varepsilon< \lambda$
			\begin{align}
			& \sup_y \left\{ \overline{F}(\gamma(y)) - f_0(\gamma(y)) \right\} \label{eqn:establish_subsol_target} \\
			& \leq \sup_y \left\{ \left(\overline{F}(\gamma(y)) - \epsilon\left(\frac{\overline{F}(\gamma(y)) - h(\gamma(y))}{\lambda}\right) \right) - \left(f_0(\gamma(y)) - \epsilon g_0(y) \right) \right\} < \infty, \notag
			\end{align}
			as we can replace $\overline{F}$ by its upper semi-continuous regularization $\overline{f}$ first on the right and then on the left-side of the inequality. Note that the lemma indeed suffices to establish the sub-solution property as for any function $\phi$ on $X$, we have $\sup_x \phi(x) = \sup_y \phi(\gamma(x))$ because $\gamma$ is surjective.
			
			\bigskip

			We extend the proof of Proposition \ref{proposition:identify_resolvent}. In that proof, we used that the pseudo-resolvent is contractive. In this case, we have to pass to the limit using the adapted notion of convergence. Thus, we replace contractivity by strict equi-continuity on bounded sets, Condition \ref{condition:convergence_of_generators_and_conditions_extended_supsuperlim} \eqref{item:convH_strict_equicont_resolvents}.
			
			\smallskip

			For every $n \geq 1$ set $f_n := R_n(\lambda) h_n$ and $g_n :=  \frac{f_n - h_n}{\lambda}$. Note that $f_n$ is well defined as $h_n \in B_n$. By assumption there are $(f_{n,0},g_{n,0}) \in H_{n,\dagger}$ such that \eqref{eqn:convergence_condition_sublim_constants}, \eqref{eqn:convergence_condition_sublim_uniform_gn}  and \eqref{eqn:sublim_generators_upperbound} are satisfied. Let $\varepsilon \in (0,\lambda)$ and define
			\begin{equation} \label{eqn:defs_of_h_n}
			h_n^{\varepsilon}  := f_n - \varepsilon g_n=  f_n - \varepsilon \frac{f_n - h_n}{\lambda}, \qquad h_{n,0}^{\varepsilon}  := f_{n,0} - \varepsilon g_{n,0}.
			\end{equation} 
			
			Note the following
			\begin{enumerate}[(1)]
				\item \label{item:pseudoresolvent} $f_n = R_n(\varepsilon)h_{n}^{\varepsilon}$ because $R_n$ is a pseudo-resolvent;
				\item \label{item:viscosity_bound} As the domain and range of $H_{n,\dagger}$ are contained in $B_n$, we can apply $R_n(\varepsilon)$ to $h_{n,0}^{\varepsilon}$. By Condition \ref{condition:convergence_of_generators_and_conditions_extended_supsuperlim} \eqref{item:convH_pseudoresolvents_solve_HJ} and Lemma \ref{lemma:identification:viscosity_subsol} we find $f_{n,0} \geq (R_n(\varepsilon)h_{n,0}^{\varepsilon})^* \geq R_n(\varepsilon)h_{n,0}^{\varepsilon}$.
				\item Finally
				\begin{equation} \label{eqn:uniform_upper_lower_bound_on_hepsilon}
				\begin{aligned}
				& \sup_{n} h_n^{\varepsilon}(x) \leq \sup_{n} h_n(x) =: M_1 < \infty, \\
				& \inf_{n} h_{n,0}^{\varepsilon}(x) \geq \inf_{n} f_{n,0} - \varepsilon \sup_{n} g_{n,0} =: M_2 > - \infty.
				\end{aligned}
				\end{equation}
				The first equality follows by $\LIMSUP_n h_{n} \leq h$ and $\varepsilon < \lambda$, whereas the second inequality follows by \eqref{eqn:convergence_condition_sublim_constants} and \eqref{eqn:convergence_condition_sublim_uniform_gn}. Denote by $M := M_1 - M_2$.
			\end{enumerate}

			Pick $x \in X$. Pick any $q \in \cQ$ such that $x \in K^q$ and such that there are $x_n \in K_n^q$ which satisfy $\eta_n(x_n) \rightarrow \eta(x)$ in $\cX$ (There is at least one such $q$ by Assumption \ref{assumption:abstract_spaces_q2} \eqref{item:assumption_abstract_2_exists_q}). Now take $\delta' > 0$ arbitrary. By \eqref{eqn:uniform_upper_lower_bound_on_hepsilon} and Condition \ref{condition:convergence_of_generators_and_conditions_extended_supsuperlim} \eqref{item:convH_strict_equicont_resolvents}, we find $\hat{q} \in \cQ$ such that for all $n$, functions  $\phi_1,\phi_2 \in B_n$ satisfying $\sup_{y \in X_n} \phi_1(y) - \phi_2(y) \leq M$ 
			\begin{equation*}
			\sup_{y \in K_n^q} R_n(\varepsilon) \phi_1(y) - R_n(\varepsilon) \phi_2(y)  \leq \delta' M +  \sup_{y \in K_n^{\hat{q}}} \phi_1(y) - \phi_1(y) 
			\end{equation*}
			Fix $\delta = M \delta'$, which we can choose arbitrarily small by choosing $\delta'$ small. In addition, we can find $z \in K_n^{\hat{q}}$ such that 
			\begin{equation} \label{eqn:proof_bound_equi_cont}
			\sup_{y \in K_n^q} R_n(\varepsilon) \phi_1(y) - R_n(\varepsilon) \phi_2(y)  \leq \delta +  \sup_{y \in K_n^{\hat{q}}} \phi_1(y) - \phi_1(y)  \leq 2\delta +  \phi_1(z) - \phi_2(z).
			\end{equation}
			For next computation, we use \eqref{item:pseudoresolvent} and \eqref{item:viscosity_bound} in line 3, Equation \eqref{eqn:defs_of_h_n} in line 5 and for line 4 we find $z_n \in K_n^{\widehat{q}}$ such that Equation \eqref{eqn:proof_bound_equi_cont} holds for all $n$ and $h_{n}^\varepsilon,h_{n,0}^\varepsilon$ instead of $\phi_1,\phi_2$. This gives
			\begin{align}
			&  f_n(x_n) - f_{n,0}(x_n) \label{eqn:proof_conv_resolvents_main_LHS} \\
			& \quad \leq  \sup_{y \in K_n^q} f_n(y) - f_{n,0}(y) \notag \\
			& \quad \leq  \sup_{y \in K_n^q} R_n(\varepsilon) h_n^{\varepsilon}(y) - R_n(\varepsilon) h_{n,0}^{\varepsilon}(y) \notag \\
			& \quad \leq 2\delta +  h_n^{\varepsilon}(z_n) - h_{n,0}^{\varepsilon}(z_n)  \notag \\
			& \quad \leq 2\delta +  \left(f_n(z_n) - \varepsilon \frac{f_n(z_n) - h_n(z_n)}{\lambda}\right) - \left( f_{n,0}(z_n) - \varepsilon g_{n,0}(z_n)\right). \notag
			\end{align}
			Recall that our aim is to prove \eqref{eqn:establish_subsol_target}. Our next step is to take a $\limsup_n$ on both sides of the inequality. To study this limsup, we see that only the term $g_{n,0}$ is not yet understood in terms of its limiting behavior. Our aim is to apply \eqref{eqn:sublim_generators_upperbound}, for which we need to construct a subsequence $n(k)$ for which $\widehat{\eta}_{n(k)}(z_{n(k)}) \rightarrow \widehat{\eta}(y)$ in $\cY$ with $y \in \widehat{K}^{\hat{q}}$ satisfying $f_{n(k),0}^{\hat{q}}(z_{n(k)}) \rightarrow f_0(\gamma(y))$. 
			
			Without loss of generality, we assume that
			\begin{equation*} 
			\limsup_n f_n(x_n) - f_{n,0}(x_n) > - \infty,
			\end{equation*}
			as there is nothing to prove otherwise.	Again without loss of generality, we restrict to a subsequence $n(k)$ of $n$ such that
			\begin{itemize}
				\item the $\limsup$ on the left-hand side of \eqref{eqn:proof_conv_resolvents_main_LHS}  is achieved as a limit:
				\begin{equation} \label{eqn::limsup_achieved}
				\limsup_n f_n(x_n) - f_{n,0}(x_n) = \lim_k f_{n(k)}(x_{n(k)}) - f_{n(k),0}(x_{n(k)}) > - \infty,
				\end{equation}
				\item there is some $y \in \widehat{K}^{\hat{q}}$ with $\lim_k \hat{\eta}_{n(k)}(z_{n(k)}) = \hat{\eta}(y)$. This is possible due to Definition \ref{assumption:abstract_spaces_q2} \eqref{item:assumption_abstract_2_limit_compact}. Note that by continuity of $\hat{\gamma} : \cY \rightarrow \cX$ also $\lim_k \eta_{n(k)}(z_{n(k)}) = \eta(\gamma(y))$ with $\gamma(y) \in K^q$ (Assumption \ref{assumption:abstract_spaces_q2} \eqref{item:assumption_abstract_2_gamma_mapsintoeachother}).
			\end{itemize}

			We first consider the $\liminf_k$ over both sides of the inequality \eqref{eqn:proof_conv_resolvents_main_LHS}. By our assumption \eqref{eqn::limsup_achieved} on $x_{n(k)}$ we find 
			\begin{equation*} 
			\liminf_k \left(f_n(z_n) - \varepsilon \frac{f_n(z_n) - h_n(z_n)}{\lambda}\right) - \left( f_{n,0}(z_n) - \varepsilon g_{n,0}(z_n)\right) > - \infty.
			\end{equation*}
			By \eqref{eqn:upper_bound_resolvents} the sequences $\{f_n\}_{n \geq 1}$ and $\{h_n\}_{n \geq 1}$ are uniformly bounded from above and by assumption \eqref{eqn:convergence_condition_sublim_uniform_gn} we have a uniform upper bound on $\{g_{n,0}\}_{n \geq 1}$. This leads to 
			\begin{equation*}
			\limsup_k f_{n(k),0}(z_{n(k)}) < \infty.
			\end{equation*}
			By \eqref{eqn:convergence_condition_sublim_constants}, choosing $c$ larger than this $\limsup$, we find $\lim_k  f_{n(k),0}(z_{n(k)}) = f_0(\gamma(y))$, which established also the condition for the application of \eqref{eqn:sublim_generators_upperbound}. Taking the $\limsup_k$ over both sides of \eqref{eqn:proof_conv_resolvents_main_LHS}, we find
			\begin{multline*}
			\limsup_n f_n(x_n) - f_{n,0}(x_n) = \lim_k f_{n(k)}(x_{n(k)}) - f_{n(k),0}(x_{n(k)}) \\
			\leq 2 \delta + \left(\overline{F}(\gamma(y))- \varepsilon \frac{\overline{F}(\gamma(y)) - h(\gamma(y))}{\lambda}\right) - \left( f_{0}(\gamma(y)) - \varepsilon g_{0}(y)\right).
			\end{multline*}
			Now we take the supremum over $y$ on the right-hand side. Afterwards, we send $\delta \rightarrow 0$. This gives the correct right-hand side for \eqref{eqn:establish_subsol_target}. Next, we work on the left-hand side. We take a supremum over all $q$ and sequences $\eta_n(x_n) \rightarrow \eta(x)$ in $\cX$ with $x_n \in K_n^q$, followed by a supremum over $x$ This establishes \eqref{eqn:establish_subsol_target} which concludes the proof.
		\end{proof}

	\begin{proof}[Proof of Theorem \ref{theorem:convergence_of_resolvents}]
		Fix $\lambda > 0$, $h \in D$ and $h_n \in B_n$ with $\LIM h_n = h$. Let $\overline{F}$ and $\underline{F}$ be as in Proposition \ref{proposition:existence_viscosity_sub_super_solutions}. By construction, it follows that $\overline{F} \geq \underline{F}$. By the comparison principle, we also have $\underline{F}_* \leq \overline{F}^*$. The combination of these inequalities yields $\underline{F} = \underline{F}_* = \overline{F} = \overline{F}^*$. Denote this function by $\hat{R}(\lambda) h$, which is therefore the unique viscosity solution to \eqref{eqn:pair_of_HJ_approx_theorem}.
		
		\smallskip
		
		Following Proposition \ref{proposition:extending_operators_via_LIM}, define
		\begin{equation*}
		\cD(R(\lambda)) = \left\{h \in C_b(X) \, \middle| \, \exists \, h_n \in B_n: \LIM h_n = h, \, \LIM R_n(\lambda) h_n \text{ exists and is continuous}  \right\}
		\end{equation*}
		and $R(\lambda)h = \LIM R_n(\lambda) h_n$. By the argument above, $R(\lambda)$ extends $\hat{R}(\lambda)$. By Proposition \ref{proposition:extending_operators_via_LIM} $\cD(R(\lambda))$ is quasi-closed, and as $D$ is quasi-dense in $C_b(X)$ by assumption we find $\cD(R(\lambda)) = C_b(X)$. In addition Proposition \ref{proposition:extending_operators_via_LIM} $\cD(R(\lambda))$ yields that $R(\lambda)$ is strictly continuous on bounded sets.
		
		The pseudo-resolvent property follows by continuity from that of the operators $R_n$. The local strict continuity on bounded sets can be proven by making the estimates in the proof of Proposition \ref{proposition:extending_operators_via_LIM} (a) uniform for $\lambda$ with $0 < \lambda \leq \lambda_0$, using the uniform choice of $\hat{q}$ as in Condition \ref{condition:convergence_of_generators_and_conditions_extended_supsuperlim} \eqref{item:convH_strict_equicont_resolvents}.
	\end{proof}

	\section{Convergence of semigroups} \label{section:convergence_of_semigroups}

	\begin{theorem} \label{theorem:CL_extend}
		Let Condition \ref{condition:convergence_of_generators_and_conditions_extended_supsuperlim} be satisfied. Suppose in addition that for all $n$ we have a collection of functions $B_n$ such that: $C_b(X_n) \subseteq B_n \subseteq M(X_n)$ and suppose that $R_n(\lambda) C_b(X_n) \subseteq C_b(X_n)$.
		
		\smallskip
		
		 Let $D \subseteq C_b(X)$ be quasi-dense in $C_b(X)$. Suppose that for each $\lambda > 0$ and $h \in D$ the comparison principle holds for 
		\begin{equation*} 
		f - \lambda H_\dagger f = h, \qquad f - \lambda H_\ddagger f = h.
		\end{equation*}
		Denote by $R(\lambda) : C_b(X) \rightarrow C_b(X)$ the operators constructed in Theorem \ref{theorem:convergence_of_resolvents}
		
		Consider the operators
		\begin{align} 
		\widehat{H}_n & := \bigcup_{\lambda} \left\{\left(R_n(\lambda)h, \frac{R_n(\lambda) h - h}{\lambda}\right) \, \middle| \, h \in C_b(X_n) \right\}, \label{def:hatHn} \\
		\widehat{H} & := \bigcup_{\lambda} \left\{\left(R(\lambda)h, \frac{R(\lambda) h - h}{\lambda}\right) \, \middle| \, h \in C_b(X) \right\}, \label{def:hatH}
		\end{align}
		as in Proposition \ref{proposition:Hhat_properties}.
		
		\smallskip
		
		Let $V_n(t)$ and $V(t)$ be the operator semigroups on the uniform closures of $\cD(\widehat{H}_n)$ and $\cD(\widehat{H})$ generated by $\widehat{H}_n$ and $\widehat{H}$ as in the Crandall-Ligget theorem, see Theorem \ref{theorem:CL}. 
		
	Suppose that the semigroups $V_n(t)$ are strictly equi-continuous on bounded sets: for all $q \in \cQ$, $\delta > 0$ and $t_0 > 0$, there is a $\hat{q} \in \cQ$ such that for all $n$ and $h_{1},h_{2} \in B_n$ and $0 \leq t \leq t_0$ that
		\begin{equation*}
		\sup_{y \in K_n^q} \left\{ V_n(t)h_{1}(y) - V_n(t)h_{2}(y) \right\} \leq \delta \sup_{x \in X_n} \left\{ h_{1}(x) - h_{2}(x) \right\} + \sup_{y \in K^{\hat{q}}_n} \left\{ h_{1}(y) - h_{2}(y) \right\}.
		\end{equation*}
		
		Denote by $\cD$ and the quasi-closure of the uniform closure of $\cD(\widehat{H})$.
		Then
		\begin{enumerate}[(a)]
			\item \label{item:semigroup_density_domains} We have $\widehat{H} \subseteq ex-\LIM \widehat{H}_n$ as in Definition \ref{definition:abstract_LIM}. That is, for all $(f,g) \in \widehat{H}$ there are $(f_n,g_n) \in \widehat{H}_n$ such that $\LIM f_n = f$ and $\LIM g_n = g$.
			\item \label{item:semigroup_extensionV} The semigroup $V(t)$ extends to the quasi-closure $\cD$ of $\cD(\widehat{H})$ on which it is locally strictly equi-continuous on bounded sets.
			\item  \label{item:semigroup_approx_of_domain} For each $f \in \cD$ there are $f_n$ in the uniform closures of $\cD(\widehat{H}_n)$ such that $\LIM f_n = f$.
			\item \label{item:semigroup_convergence_semigroups}  If $f_n$ are in the uniform closures of $\cD(\widehat{H}_n)$ and $f \in \cD$ such that $\LIM f_n =f$ and $t_n \rightarrow t$ then $\LIM V_n(t_n)f_n = V(t) f$. 
		\end{enumerate}
	\end{theorem} 


	\begin{remark}
		For applications it is of interest to know whether $\cD(\widehat{H})$ is quasi-dense in $C_b(X)$. If the resolvent is obtained as in Theorem \ref{theorem:convergence_of_resolvents} and $\LIM R_n(\lambda) h = h$ as $\lambda \downarrow 0$ for all $n$, then this can sometimes be established directly from an approximation procedure, see for example Lemma 7.19 in \cite{FK06}. This is indeed what one would expect from a Crandall-Liggett theorem for the strict topology.
		Another possibility is to find an explicit expression for the resolvent and verify this property directly. 
		We will pursue a third possibility below, see Proposition \ref{proposition:zero_operator}, that is based on a comparison principle.
	\end{remark}
	
	The main step to go from the result of Theorem \ref{theorem:convergence_of_resolvents} to that of above theorem is an approximation argument by Kurtz: Theorem 3.2 of \cite{Ku73}. The key argument in the approximation result is the embedding of all spaces and semigroups in a common product space. The notion of $\LIM$ is embedded into this product space as a closed subspace. We study these spaces in next proposition.

	\begin{proposition} \label{proposition:study_LIM}
		Let Assumptions \ref{assumption:abstract_spaces} and \ref{assumption:abstract_spaces_q} be satisfied. The space
		\begin{equation*}
		\fL  := \left\{\ip{f}{\{f_n\}} \, \middle| \, f_n \in M_b(X_n), f \in M_b(X), \sup_{n} \vn{f_n} < \infty \right\}, 
		\end{equation*}
		equipped with the norm $\vn{\ip{f}{\{f_n\}}} = \vn{f} \vee \sup_{n} \vn{f_n}$ is a Banach space. Set
		\begin{equation*}
		\fP := \left\{(\ip{f}{\{f_n\}},f) \in \fL\times M_b(X) \, \middle| \, \LIM f_n = f \right\}.
		\end{equation*}
		The set $\fP$ is a closed linear subspace $\fL \times M_b(X)$ and $\fP$ interpreted as an operator from $\fL$ to $M_b(X)$ satisfies $\vn{\fP} \leq 1$. 
	\end{proposition}
	
	In the proposition, we do not consider for which $f$ there are $f_n$ such that $f = \LIM f_n$. We assume this e.g. in Condition \ref{condition:convergence_of_generators_and_conditions_extended_supsuperlim} \eqref{item:approximation_of_functions}. In particular cases, however, surjectivity of $\fP$ can be established directly.
	
	\begin{lemma} \label{lemma:construction_f_n_that_converge_to_f}
		Suppose that $\cX$ is a normal space and that the maps $\eta_n : X_n \rightarrow \cX$ are continuous and that $\eta : X \rightarrow \cX$ is a homeomorphism onto its image. Then for each $f \in M_b(X)$ there are $f_n \in M_b(X_n)$ such that $\LIM f_n = f$. If $f \in C_b(X)$, then $f_n$ can be chosen in $C_b(X_n)$.
		
	\end{lemma}
	
	\begin{proof}
		First let $f \in C_b(X)$. The function $g := f \circ \eta^{-1}$ is a continuous function with norm $\vn{f}$ on $\eta(X) \subseteq \cX$.  By the Tietze extension Theorem, it extends to a continuous function $g$ on $\cX$ with norm $\vn{g} = \vn{f}$. We then define $f_n := g \circ \eta_n$, which leads to $\vn{f_n} \leq \vn{f}$. Next, let $x_n \in K_n^q$ and $x \in K^q$ such that $\eta_n(x_n) \rightarrow \eta(x)$. Because $g$ is continuous, we find that $g(\eta_n(x_n)) \rightarrow g(\eta(x))$ implying that $f_n(x_n) \rightarrow f(x)$. Thus $\LIM f_n = f$. The result for $M_b(X)$ then follows from the monotone class theorem, see e.g. Theorem 2.12.9 in \cite{Bo07}.
	\end{proof}

	\begin{remark} \label{remark:monotone_class_theorem}
		
		Note that in the use of the Monotone class theorem, we establish the result for functions that are bounded and measurable with respect to the $\sigma$-algebra generated by all bounded and continuous functions. For a general topological space this implies the final result holds for the set of bounded and measurable functions with respect to the Baire $\sigma$-algebra. In the case that $\cX$ is Polish, the Baire and Borel $\sigma$ algebra's coincide. More general, this holds for perfectly normal spaces, see Proposition 6.3.4 in \cite{Bo07}. 
	\end{remark}
	
	\begin{proof}[Proof of Proposition \ref{proposition:study_LIM}]
		That $\fL$ is a Banach space, as that $\fP$ is linear, is immediate. We establish that $\fP$ is norm closed. Let $f_{n}^{k}, f_n \in M_b(X_n)$ and $f^k, f \in M_b(X)$ such that for all $k$: $\LIM f_n^{k} = f^k$ and $\lim_k \left(\vn{f-f^k} \vee \sup_{n} \vn{f_n^{k} - f_n}\right) = 0$. We prove $\LIM f_n = f$. 
		
		First of all, let $k$ be such that $\sup_{n} \vn{f_n^k - f_n} \leq 1$. Then $\vn{f_n} \leq \vn{f_n - f_n^{k}} + \vn{f_n^{k}} \leq 1 + \vn{f_n^k}$. The final term is bounded as $\LIM f_n^{k} = f^k$. For the second property, fix $q \in \cQ$ and $x_n \in K_n^q$ converging to $x \in K^q$. We have
		\begin{align*}
		\left|f_n(x_n) - f(x)\right| & \leq \left| f_n(x_n) - f_n^{k}(x_n) \right| + \left| f_n^{k}(x_n) - f^k(x)\right| +  \left| f^k(x) - f(x) \right| \\
		& \leq \sup_m \vn{f_m - f_m^{k}}  + \left| f_n^{k}(x_n) - f^k(x)\right| + \vn{f^k - f}.
		\end{align*}
		The first and third term on the right-hand side can be made arbitrarily small by choosing $k$ large. For fixed $k$ the term final term converges to $0$ as $\LIM_n f_n^k = f^k$. Thus, we find $f_n(x_n) \rightarrow f(x)$. Contractivity of $\fP$ follows by assumption.
	\end{proof}

		The proof of the theorem is based on a general semigroup approximation result \cite[Theorem 3.2]{Ku73}.
		
		\begin{proof}[Proof of Theorem \ref{theorem:CL_extend}]
			
			For the proof of \eqref{item:semigroup_density_domains}, pick $f \in \cD(\widehat{H})$. By definition there are $\lambda > 0$ and $h \in D$ such that $f = R(\lambda) h$. By the assumption in Theorem \ref{theorem:convergence_of_resolvents}, there are $h_n \in B_n \cap M_b(X_n)$ such that $\LIM h_n = h$. By Theorem \ref{theorem:convergence_of_resolvents}, we obtain $\LIM R_n(\lambda) h_n = R(\lambda) h = f$. By construction $R_n(\lambda) h_n \in \cD(\widehat{H}_n)$ establishing \eqref{item:semigroup_density_domains}.
			
			\smallskip
			
			We proceed with the proof of \eqref{item:semigroup_extensionV}, \eqref{item:semigroup_approx_of_domain} and \eqref{item:semigroup_convergence_semigroups} for which we will use Theorem 3.2 of \cite{Ku73}. Recall the set $\fL$ and the closed subset $\fP$ of Proposition \ref{proposition:study_LIM}. Denote also
			\begin{equation*}
			\cH := \left\{(\ip{f}{\{f_n\}},\ip{g}{\{g_n\}}) \in \fL \times \fL \, \middle| \, (f_n,g_n) \in \widehat{H}_n, (f,g) \in \widehat{H}\right\}.
			\end{equation*}
			Note that $\cH$ is dissipative and satisfies the range condition because the operators $\widehat{H}_n$ and $\widehat{H}$ do as well. The semigroup $\cV(t)$ generated by $\cH$ equals $\cV(t)\left(\ip{f}{\{f_n\}}\right) = \ip{V(t)f}{\{V_n(t)f_n\}}$ on the uniform closure in $\fL$ of $\cD(\widehat{H}) \times \prod_n \cD(\widehat{H}_n)$ (which might be smaller than the product over the uniform closures).
			
			By \eqref{item:semigroup_density_domains}, we have 
			\begin{equation*}
			\widehat{H} = \left\{(f,g) \, \middle| \, ((\ip{f}{\{f_n\}},f), (\ip{g}{\{g_n\}},g)) \in \cH \cap (\cD(\fP) \times \cD(\fP))  \right\},
			\end{equation*}
			so that all conditions for Theorem 3.2 of \cite{Ku73} are satisfied. From Equation (3.4) in \cite{Ku73}, we infer that if $t \geq 0$, $f_n \in \cD(\widehat{H}_n)$ and $f \in \cD(\widehat{H})$ such that $\LIM f_n = f$, then $\LIM V_n(t) f_n = V(t)f$.
			
			Fix $t > 0$ and define
			\begin{multline} \label{eqn:extension1}
			\cD(V(t)) := \\
			\left\{h \in C_b(X) \, \middle| \, \exists \, h_n \in B_n: h = \LIM h_n, \LIM V_n(t) h_n \text{ exists and is continuous}  \right\}.
			\end{multline}
			By the argument above $\cD(V(t))$ contains $\cD(\widehat{H})$. By Proposition \ref{proposition:extending_operators_via_LIM} the set $\cD(V(t))$ is quasi-closed and the operator $V(t)$ extends to $\cD(V(t))$ on which it is strictly continuous on bounded sets.
			
			Thus, the quasi-closure $\cD$ of $\cD(\widehat{H})$ is contained in $\cD(V(t))$ for all $t$. Thus, \eqref{item:semigroup_extensionV} , \eqref{item:semigroup_approx_of_domain} and \eqref{item:semigroup_convergence_semigroups} (for $t_n = t$) all follow from Proposition \ref{proposition:extending_operators_via_LIM}.
			
			We now extend \eqref{item:semigroup_convergence_semigroups} to the context of $t_n$ converging to $t$. Thus, let $t_n \rightarrow t$, $f_n \in \cD(\widehat{H}_n)$ and $f \in \cD(\widehat{H})$ such that $\LIM f_n = f$. We have seen above that $\LIM V_n(t)f_n = V(t)f$. Using the decomposition 
			\begin{equation*}
			V(t)f - V_n(t_n)f_n = \left[ V(t)f - V_n(t)f_n \right] + \left[V_n(t)f_n - V_n(t_n)f_n \right]
			\end{equation*}
			and the uniform continuity of $\cV(t)$ on $\cD(\widehat{H}) \times \prod_n \cD(\widehat{H}_n)$ we find that also $\LIM V_n(t_n)f_n = V(t)f$.
			
			Repeating the argument above for
			\begin{multline} \label{eqn:extension2}
			\cD_{\{t_n\}}(V(t)) :=\\
			 \left\{h \in C_b(X) \, \middle| \, \exists \, h_n \in B_n: h = \LIM h_n, \LIM V_n(t_n) h_n \text{ exists and is continuous}  \right\},
			\end{multline}
			we find by Proposition \ref{proposition:extending_operators_via_LIM} a second extension of $V(t)$ with the correct properties. However, as we have seen the extensions based on \eqref{eqn:extension1} and \eqref{eqn:extension2} agree on the quasi-dense subset $\cD(\widehat{H}) \subseteq \cD$ and therefore must be the same on $\cD$. This establishes \eqref{item:semigroup_convergence_semigroups}.
		\end{proof}

\section{Density of the domain} \label{section:density_of_domain}

	In Theorem \ref{theorem:CL_extend}, we obtained a semigroup that was defined on the quasi-closure of $\cD(\widehat{H})$. In applications, often it is of interest to know whether this quasi-closure is in fact equal to $C_b(X)$. 
	
	A key method to verify this quasi-density, is the verification that as $\lambda \downarrow 0$, we have $\LIM R(\lambda) h = h$ for the buc topology. For this there are two possible strategies:
	\begin{enumerate}
		\item One finds a explicit characterization of $R$, i.e. a control representation, and verifies this property directly,
		\item In the context of Theorem \ref{theorem:convergence_of_resolvents}, one knows that  $\lambda \downarrow 0$, we have $\LIM R_n(\lambda) h = h$ for each $n$ and $h$ and establishes that such statements can be lifted to the limit, see e.g. Lemma 7.19 in \cite{FK06}.
	\end{enumerate}

We will introduce a new method, that bootstraps the procedure of Section \ref{section:convergence_of_operators_and_inverses}.

We proceed with an informal discussion. Consider the setting in which $R(\lambda)h$ is the viscosity solution to $f - \lambda H f = h$. In the linear theory, it is generally known that $R(\lambda)h \rightarrow h$ as $\lambda \downarrow 0$ in an appropriate topology. Indeed, as $R(\lambda)h \in \cD(H)$, this establishes density of $\cD(H)$. We expect the same result to hold true in the non-linear case. Consider the operators $A_n = \frac{1}{n} H$ and resolvents $\cR_n(\lambda) = R\left(\tfrac{\lambda}{n}\right)$. Formally, the operator $A_n$ converges to the zero-operator $0 \cdot H$, so that we expect that the relaxed $\limsup$ and $\liminf$ of $\cR_n(1)h$ yield a viscosity sub- and supersolution to $f - 0 \cdot H f = h$, or informally written, to $f = h$. Clearly, we expect these limits to equal $h$. To obtain this result rigorously, we need a comparison principle.

Informally, we need that $\cD(H)$ that is `large enough to uniquely identify functions'. 

We make this intuition rigorous.

\begin{proposition} \label{proposition:zero_operator}
	Let $X$ be a space with metrizable compact sets and let $H_\dagger \subseteq C_b(X) \times C_b(X)$ and $H_\ddagger \subseteq C_b(X) \times C_b(X)$. Let $R(\lambda) : C_b(X) \rightarrow C_b(X)$ be a collection of operators that is locally strictly equi-continuous on bounded sets as in Definition \ref{definition:local_strict_equicont_of_pseudoresolvent}. 
	
	Let $D$ be a quasi-dense set in $C_b(X)$ and suppose that for $\lambda > 0$ and $h \in D$, the function $R(\lambda)h$ is a viscosity solution to
	\begin{equation*}
	f - \lambda H_\dagger f = h, \qquad f - \lambda H_\ddagger f = h.
	\end{equation*}
		
	Denote $A_{\dagger} := 0 \cdot \cH_\dagger$ and $A_{\ddagger} := 0 \cdot \cH_\ddagger$. Suppose that the comparison principle holds for $f - A_\dagger f = h_1$ and $f - A_\ddagger f = h_2$ for $h_1,h_2 \in D$. Let $\lambda_n \downarrow 0$. Then for all $h \in D$ we have $\LIM_n R(\lambda_n)h = h$. In particular, the domain $\cD(\widehat{H})$ is quasi-dense in $C_b(X)$.
\end{proposition}

\begin{remark}
	Generally, proofs that establish the comparison principle for $f - \lambda H_\dagger f = h$ and $f - \lambda H_\ddagger f = h$ can be adapted in a straightforward way to also establish the comparison principle for $f - \lambda A_\dagger f = h$ and $f - \lambda A_\ddagger f = h$.
\end{remark}

We start by proving the seemingly trivial fact that $h$ solves the equation $f = h$.

\begin{lemma} \label{lemma:trivial_viscosity_solutions}
	Let $H_\dagger \subseteq LSC_l(X) \times USC_u(X)$ and $H_\ddagger \subseteq USC_u(X) \times LSC_l(X)$ and define $A_{\dagger} := 0 \cdot \cH_\dagger$ and $A_{\ddagger} := 0 \cdot \cH_\ddagger$.
	
	For any $h \in C_b(X)$ the function $h$ is a subsolution to $f -  A_{\dagger} f = h$ and a supersolution to $f - A_{\ddagger}f = h$.
	
\end{lemma}

\begin{proof}
	We show that $h$ is a viscosity subsolution to  $f -  A_{\dagger} f = h$ and a supersolution to $f - A_{\ddagger}f = h$. 
	
	We establish that $h$ is a viscosity subsolution of $f - A_\dagger f = h$ by using Lemma \ref{lemma:technical_to_establish_visc_sol} (a). Let $(f,g) \in A_\dagger$. Thus $g = 0 \cdot \hat{g}$ with $(f,\hat{g}) \in H_\dagger$. Note that $\hat{g}$ is bounded from above. Thus $0 \cdot \hat{g} \leq 0$. Thus, for all $\varepsilon > 0$ we have
	\begin{equation*}
	\sup_x h(x) - f(x) \leq \sup_x h(x) - f(x) - \varepsilon\left((h(x) - h(x)\right) -0 \cdot \hat{g}(x)).
	\end{equation*}
	This inequality, in combination with the fact that $f$ is bounded from below, implies that the condition of Lemma \ref{lemma:technical_to_establish_visc_sol} (a) is satisfied (note that the $g$'s in the Lemma and here are different). As a consequence, we find $x_n \in X$ such that
	\begin{gather*}
	\lim_n h(x_n) - f(x_n) = \sup_x h(x) - f(x) \\
	\limsup_n h(x_n) - f(x_n) - 0 \cdot \hat{g}(x_n) \leq 0,
	\end{gather*}
	that is, $h$ is a subsolution to $f - A_\dagger f = h$. Similarly, we prove that $h$ is a supersolution to $f - A_\ddagger f = h$, which concludes the proof.
\end{proof}

In the proof below, the notion of $\LIM$ refers to buc convergence. Thus, $\cQ$ is the set of compact sets $\cK$ in $X$ with $K_n^\cK = K^\cK = \cK$. See example \ref{example:LIM_buc_convergence}. Note that we need metrizable compacts to extract converging subsequences from sequences in compact sets.

\begin{proof}[Proof of Proposition \ref{proposition:zero_operator}]
	If we can establish that $\LIM R(\lambda_n) h = h$, then we have that $\cD(\widehat{H})$ is quasi-dense in $D$. As $D$ is quasi-dense in $C_b(X)$ by assumption, this establishes the claim.
	
	\smallskip
	
	By Lemma \ref{lemma:trivial_viscosity_solutions} and uniqueness of viscosity solutions, it suffices to apply Proposition \ref{proposition:existence_viscosity_sub_super_solutions} for $A_{n,\dagger} = \lambda_n H_{\dagger}$, $A_{n,\ddagger} := \lambda_n H_{\ddagger}$ and $\cR_n := R\left(\lambda_n\right)$.
	
	Doing so, we obtain that $\cR_n h$ is a viscosity subsolution to $f - A_{n,\dagger}f = h$ and a supersolution to $f -  A_{n,\ddagger} f = h$. Thus, it suffices to verify Condition \ref{condition:convergence_of_generators_and_conditions_extended_supsuperlim}.

	\smallskip
	
	We work with $B_n = C_b(X)$. Then \eqref{item:approximation_of_functions} is immediate. Condition \ref{item:convH_pseudoresolvents_solve_HJ} follows by assumption. Also \eqref{item:convH_strict_equicont_resolvents} is immediate by local strict equi-continuity on bounded sets of the resolvent $R(\lambda)$.

	\smallskip
	Next, we establish \eqref{item:convH_sub_superLIM}, i.e.:
	\begin{align*}
	A_\dagger \subseteq ex-\LIMSUP A_{n,\dagger}, \qquad A_{\ddagger} \subseteq ex-\LIMINF A_{n,\ddagger}.
	\end{align*}
	We only prove the first claim. Suppose $(f,g) \in A_{\dagger}$. Then there is a $\hat{g}$ such that $g = 0 \cdot \hat{g}$ and $(f,\hat{g}) \in H_\dagger$. Set $f_n = f$ and $g_n = \lambda_n \cdot \hat{g}$. It follows that $(f_n,g_n) \in A_{n,\dagger}$. It is immediate that $\LIM f_n \wedge c = f$ for all $c$ and as $\hat{g}$ is bounded above also $\sup_n \sup_x g_n(x) \leq 0 \vee \sup_x g(x) < \infty$. Finally, we establish \eqref{eqn:sublim_generators_upperbound}. Note that as $x_n \rightarrow x$ in some compact set $\cK \subseteq X$, we have
	\begin{equation*}
	\limsup_n \widehat{g}(x_n) \leq \widehat{g}(x)
	\end{equation*}
	as $\widehat{g}$ is upper semi-continuous. It follows that
	\begin{equation*}
	\limsup_n g_n(x_n) = \limsup_n \lambda_n \widehat{g}(x_n) \leq 0 \cdot g(x).
	\end{equation*}
	Thus, we conclude by Proposition \ref{condition:convergence_of_generators_and_conditions_extended_supsuperlim} and the comparison principle that $\LIM \cR_n h$ as $n \rightarrow \infty$ is the unique viscosity solution $h$ to $f -  A_{\dagger} f = h$ and $f - A_{\ddagger}f = h$.
\end{proof}

\appendix

\section{Proof of Proposition \ref{proposition:buc_cont_equiv_to_double_bound}} \label{appendix:proof_of_topology_implications}

\begin{proof}[Proof of Proposition \ref{proposition:buc_cont_equiv_to_double_bound}]
	We prove (a) to (b). Fix a compact set $K \subseteq X$ and $r, \delta > 0$.
	
	Because the semi-norm $p(f) = \sup_{x \in K} |f(x)|$ is continuous for the strict topology. and $T$ is strictly continuous, there is a semi-norm $q(f) = \sup_n a_n \sup_{x \in K_n} |f(x)|$ such that $p(Tf - Tg) \leq q(f-g)$. 
	
	Set $C_0(r) = 2r$, $C_1(\delta,r) = a_1$, and 
	\begin{equation*}
	\hat{K}(K,\delta,r) := \bigcup_{i : a_i > \delta} K_i.
	\end{equation*}
	As $a_i \downarrow 0$, $\hat{K}$ is indeed a compact set. Let $n_0$ such that for $n \geq n_0$ we have $a_n  \leq \delta$. Then, if $\vn{f} \vee \vn{g} \leq r$:
	\begin{align*}
	p(Tf - Tg) \leq q(f-g)  & \leq \sup_{n < n_0} a_n \sup_{x \in K_n} |f(x) - g(x)| + \sup_{n \geq n_0} a_n \sup_{x \in K_n} |f(x) - g(x)| \\
	& \leq a_1 \sup_{x \in \hat{K}} |f(x) - g(x)|  + \delta \vn{f-g} \\
	& \leq C_1(\delta,r) \sup_{x \in \hat{K}} |f(x) - g(x)| + \delta C_0(r).
	\end{align*}
	establishing (b).
	
	We prove (b) to (c). Let $f_\alpha$ be a bounded net that converges to $f$. To prove that $Tf_\alpha$ converges strictly to $f$, we need to establish that $Tf_\alpha$ is bounded and that for each compact set $K \subseteq X$ and $\varepsilon > 0$ there is a $\alpha_0$ such that for $\alpha \geq \alpha_0$, we have
	\begin{equation*}
	\sup_{x \in K} \left|Tf_\alpha(x) - Tf(x)\right| \leq \varepsilon.
	\end{equation*}	
	First of all, as the net is bounded there is some $r$ such that $\sup_n \vn{f_n} \leq r$ for some $r$. By (b), using $\delta = 1$, we find that $\vn{Tf_\alpha - Tf} \leq C_0(r) + 2r C_1(1,r)$. Next, fix a compact set $K \subseteq X$ and $\varepsilon > 0$. By (b), with $\delta = \tfrac{1}{2}\varepsilon C_0(r)^{-1}$, we find a compact set $\widehat{K}$ such that
	\begin{equation*}
	\sup_{x \in K} \left|T f_\alpha(x) -Tf(x)\right| \leq \frac{1}{2} \varepsilon + C_1(\delta,r) \sup_{x \in \widehat{K}} \left|f_\alpha(x) - f(x)\right|
	\end{equation*}
	Thus, there is some $\alpha_0$ such that for $\alpha \geq \alpha_0$ the left hand side is bounded by $\varepsilon$. We conclude that $T$ is strictly continuous on bounded sets.
\end{proof}

\bibliographystyle{abbrv}
\bibliography{../KraaijBib}{}
\end{document}